\documentclass{conm-p-l}
\usepackage{amssymb}
\usepackage{amsthm}
\usepackage{graphicx, color, float}
\usepackage{tikz-cd}
\def\bbR{\mathrm{I\!R}}

\def\dz{\mathcal{D}}
\def\ez{\mathcal{E}}
\def\hz{\mathcal{H}}

\def\zz{\mathcal{Z}}

\def\rp{\bbR\hn^n}

\def\rtr{\bbR\nh^3}

\def\px{p}
\def\lx{\ell}

\def\hyp{\hskip.5pt\vbox
{\hbox{\vrule width2.5ptheight0.5ptdepth0pt}\vskip2pt}\hskip.5pt}
\def\hs{\hskip.7pt}
\def\hh{\hskip.4pt}
\def\nh{\hskip-.7pt}
\def\nnh{\hskip-1.5pt}
\def\hn{\hskip-.4pt}
\def\w{^{\phantom i}}
\def\txm{{T\hskip-2.9pt_x\w\hn M}}

\def\my{\mu}
\def\sy{\sigma}

\def\nav{\nabla\hskip-2.6pt_v\w\hs}
\def\naw{\nabla\hskip-2.6pt_w\w\hs}

\def\bbC{{\mathchoice {\setbox0=\hbox{$\displaystyle\mathrm{C}$}
\hbox{\hbox to0pt{\kern0.4\wd0\vrule height0.9\ht0\hss}\box0}} 
{\setbox0=\hbox{$\textstyle\mathrm{C}$}\hbox{\hbox 
to0pt{\kern0.4\wd0\vrule height0.9\ht0\hss}\box0}} 
{\setbox0=\hbox{$\scriptstyle\mathrm{C}$}\hbox{\hbox 
to0pt{\kern0.4\wd0\vrule height0.9\ht0\hss}\box0}} 
{\setbox0=\hbox{$\scriptscriptstyle\mathrm{C}$}\hbox{\hbox 
to0pt{\kern0.4\wd0\vrule height0.9\ht0\hss}\box0}}}} 

\def\dimr{\dim_{\hskip.4pt\bbR\hskip-1.2pt}^{\phantom i}}

\def\ig{parallel}
\def\Ig{Parallel}
\def\igy{parallelism}
\def\Igy{Parallelism}

\newtheorem{theorem}{Theorem}[section]
\newtheorem{lemma}[theorem]{Lemma}

\theoremstyle{definition}

\theoremstyle{remark}
\newtheorem{remark}[theorem]{Remark}

\numberwithin{equation}{section}

\begin{document}

\title[Parallel differential forms]{\Ig\ differential forms of codegree two,\\
and three-forms in dimension six}
\author[A. Derdzinski]{Andrzej Derdzinski}
\address[Andrzej Derdzinski]{Department of Mathematics\\
The Ohio State University\hskip-1pt\\
231 \hbox{W\hskip-1pt.} 18th Avenue\\
Columbus, OH 43210, USA}
\email{andrzej@math.ohio-state.edu}
\author[P. Piccione]{Paolo Piccione}
\address[Paolo Piccione]{Department of Mathematics\\
School of Sciences\\
Great Bay University\\
Dongguan, Guangdong 523000, China}
\address{{\it Permanent address\hh}: 
Departamento de Matem\'atica\\
Instituto de Matem\'atica e Esta\-t\'\i s\-ti\-ca\\
Uni\-ver\-si\-da\-de de S\~ao Paulo\\
Rua do Mat\~ao 1010, CEP 05508-900\\
S\~ao Paulo, SP, Brazil}
\email{paolo.piccione@usp.br}
\author[I.\ Terek]{Ivo Terek} 
\address[Ivo Terek]{Department of Mathematics and Statistics\\
Williams College\\
Wil\-li\-ams\-town, MA 01267, USA}
\email{it3@williams.edu}

\thanks{The\hs\ first\hs\ two\hs\ authors'\hs\ research\hs\ supported\hs\
in\hs\ part\hs\ by\hs\ FAPESP\hs\ grants\hs\ 2022/16097-2\\
and\hn\ 2022/14254-3,\nh\ the\hn\ first\hn\ author's\hn\ also\hn\ by\hn\ 
a\hn\ FAPESP\hn-\hs OSU\hn\ 2015\hn\ Regular\hn\ Research\hn\ Award\\
(FAPESP grant: 2015/50265-6).}

\begin{abstract}For a differential form on a manifold, having constant
components in suitable local coordinates trivially implies being parallel
relative to a tor\-sion-free connection, and the converse implication is
known to be true for $p$-forms in dimension $n$ when $p=0,1,2,n-1,n$. We prove
the converse for $(n-2)$-forms, and for 3-forms when $n=6$, while pointing out 
that it fails to hold for Cartan 3-forms on all simple Lie groups of 
dimensions $n\ge8$ as well as for $(n,p)=(7,3)$ and $(n,\px)=(8,4)$, where
the 3-forms and 4-forms arise in compact simply connected Riemannian 
manifolds with exceptional holonomy groups. We also provide geometric
characterizations of 3-forms in dimension six and $(n-2)$-forms in dimension
$n$ having the con\-stant-com\-po\-nents property mentioned above, and describe
examples illustrating the fact that various parts of these 
geometric characterizations are logically independent.
\end{abstract}

\makeatletter
\@namedef{subjclassname@2020}{\textup{2020} Mathematics Subject Classification}
\makeatother

\subjclass[2020]{Primary 53B05
\and
Secondary 53A55}

\keywords{Parallel differential form, locally constant differential form}

\maketitle

\setcounter{section}{0}
\setcounter{theorem}{0}
\setcounter{prop}{0}
\renewcommand{\thetheorem}{\Alph{theorem}}
\renewcommand{\theprop}{\Alph{prop}}
\section{Introduction}
\setcounter{equation}{0}
Manifolds (by definition connected) 
and tensor fields are always assumed to be smooth. 
A differential $\,\px\hs$-form $\,\my\,$ 
on an $\,n$-di\-men\-sion\-al manifold $\,M\,$ may be called
\begin{enumerate}
\item[(i)] {\it algebraically constant\/} if it has the same algebraic type
at all points,
\item[(ii)]{\it locally constant\/} when it has constant components in suitable
local coordinates around each point,
\item[(iii)] {\it \ig\/} if $\,\nabla\nnh\my=0\,$ for some
tor\-sion-free connection $\,\nabla\nnh$.
\end{enumerate}
In \cite{derdzinski-piccione-terek} our `parallel' forms are 
referred to as {\it in\-te\-gra\-ble}, while
\cite{munoz-masque-pozo-coronado-rosado-maria} uses for (ii) the term
{\it forms with constant coefficients}. With the aid of partitions of unity
one easily sees that $\,\nabla\hs$ in (iii) only needs to exist locally. Thus,
\begin{equation}\label{ifi}
\mathrm{(iii)\ follows\ from\ (ii),\ which\ is\ its\ local\ version\ with\ 
a\ flat\ connection,}
\end{equation}
if we agree to treat the items (i)--(iii) as conditions rather than
definitions, Including (\ref{ifi}), we
then have the well-known implications
\begin{equation}\label{imp}
\mathrm{(ii)}\implies\mathrm{(iii)}\implies\mathrm{(i)}\mathrm{,\ 
\ and\ \ (iii)}\implies\,(d\my=0\hh)\hh,
\end{equation}
cf.\  \cite[Prop.\,2.1]{munoz-masque-pozo-coronado-rosado-maria}, 
\cite[formula (1.1)]{derdzinski-piccione-terek}, 
as well as the equivalences
\begin{equation}\label{eqv}
\mathrm{when\
}\,\px\in\{0,1,2,n-1,n\}\mathrm{,\ \ (ii)}\iff\mathrm{(iii)}\iff\,(\my\,
\mathrm{\ is\ closed)}.
\end{equation}
See \cite[Examples\,1.5\hs-\hn1.8]{munoz-masque-pozo-coronado-rosado-maria}, 
\cite[Prop.\,D]{derdzinski-piccione-terek}. As pointed out by
Mu\~noz Masqu\'e et al.\ in 
\cite{munoz-masque-pozo-coronado-rosado-maria}, the questions of this kind
raised, but not answered, by (\ref{eqv}) start from $\,(n,\px)=(5,3)$.
Regarding
such questions, let us note here that the converse of the last implication in
(\ref{imp}) is {\it false except in the case\/} (\ref{eqv}): 
\cite[formula (1.7)]{derdzinski-piccione-terek} 
provides simple counterexamples for all $\,(n,\px)\,$ with $\,n\ge5\,$ and
$\,3\le\px\le n-2$. 

Given an algebraically constant differential form $\,\my\,$ on a manifold
and a distribution $\,\ez\hs$ naturally associated with $\,\my$,
as $\,\ez\hs$ is obviously $\,\nabla\nh$-par\-al\-lel when $\,\nabla\nnh\my=0$,
\begin{equation}\label{iii}
\mathrm{in\-te\-gra\-bi\-li\-ty\ of\ }\,\ez\hs\mathrm{\ follows\ if\
}\,\my\,\mathrm{\ happens\ to\ be\ \ig.}
\end{equation}
Our first main result, Theorem~\ref{ififf}, states that the equivalence 
(ii)$\iff$(iii) in (\ref{eqv}) remains true also when $\,\px=n-2$, as well
as for $\,(n,\px)=(6,3)$.

In general, however, (ii) implies (iii), but not conversely, with
specific counterexamples (Theorem~\ref{injec}): some related to exceptional
holonomy groups \cite{joyce}, with $\,(n,\px)=(7,3)\,$ and $\,(n,\px)=(8,4)$,
others provided by the Car\-tan 
$\,3$-forms on all simple Lie groups of dimensions $\,n\ge8$. The latter 
can be further generalized (Remark~\ref{smplg}) to all sem\-i\-simple Lie
groups without normal Lie sub\-groups of dimensions $\,3\,$ or $\,6$.

The next two results, Theorems~\ref{tfsix} and \ref{lcint}, provide geometric
characterizations of the case where $\,\my$, a $\,3$-form in dimension $\,6\,$ 
or an $\,(n-2)$-form in dimension $\,n$, 
is \ig\ (or, equivalently, locally constant, cf.\ Theorem~\ref{ififf}). The
characterizations involve 
closedness of $\,\my\,$ and, for $\,(n-2)$-forms, of a certain $\,2$-form
arising from $\,\my$, as well as in\-te\-gra\-bi\-li\-ty of an
al\-most-com\-plex structure (in the case of $\,3$-forms) and of specific
distributions naturally associated with $\,\my$.

In Sect.\,\ref{li}--\ref{dc} we describe examples showing that there are no 
redundant parts in the characterizations just mentioned. One source of these 
examples is Theorem~\ref{duncl}, dealing with the natural duality between 
nondegenerate differential \hbox{$2$-forms} $\,\sy\hs$ in (necessarily even)
dimension $\,n\,$ and
$\,(n-2)$-forms $\,\my\,$ which are in\-di\-vis\-i\-ble at each point. It 
states that,
even though $\,d\my=0\,$ whenever $\,d\sy\hn=0$, the converse implication,
obvious when $\,n\le4$, fails in all even dimensions $\,n\ge6$.

\renewcommand{\thetheorem}{\thesection.\arabic{theorem}}
\renewcommand{\theprop}{\thesection.\arabic{prop}}
\section{Preliminaries}\label{pr}
\setcounter{equation}{0}
Our convention about the exterior product of $\,1$-forms $\,\xi\hh^i$ and
vectors $\,v\nnh_j\w$ is 
\begin{equation}\label{cvt}
[\hh\xi^1\nnh\nh\wedge\ldots\wedge\hs\xi\hh^\px](v_1\w,\dots,v_\px)\,
=\,[v_1\w\wedge\ldots\wedge v_\px](\xi^1\nnh,\dots,\hs\xi\hh^\px)\,
=\,\det[\hh\xi\hh^i(v\nnh_j\w)]\hh.
\end{equation}
Given a differential $\,\px\hs$-form $\,\zeta$, with $\,\widehat{\,\,}\,$ 
meaning `delete' one has the following well-known
expression for the exterior derivative in terms of Lie brackets:
\begin{equation}\label{dbr}
\begin{array}{rrl}
[d\hh\zeta](v_0\w,\dots,v\hn_\px\w)\hskip-4pt&=&
\hskip-5pt\textstyle{\sum}_i\w(-\nnh1)^i
d_{v_i}\w[\hh\zeta(v_0\w,\dots,\widehat v_i\w,\dots,v\hn_\px\w)]\\
\hskip-4pt&
+&\hskip-5pt\textstyle{\sum}_{i,j}\w(-\nnh1)^{i+j}\zeta([v_i\w,v_j\w],
v_0\w,\dots,\widehat v_i\w,\dots,\widehat v_j\w,\dots,v\hn_\px\w)\hh,
\end{array}
\end{equation}
$v_0\w,\dots,v\hn_\px\w$ being any tangent vector fields, the summation ranges
$\,0\le i\le\px\,$ and and $\,0\le i<j\le\px$. See, e.g.,
\cite[formula (1.5a)]{besse}.

We call a vector field $\,w\,$ on a manifold {\it pro\-ject\-a\-ble along 
an in\-te\-gra\-ble distribution\/} $\,\dz\,$ if, locally, it is 
pro\-ject\-a\-ble onto local leaf spaces of $\,\dz$. As is easily seen in
coordinates such that some coordinate vector fields span $\,\dz$, 
\begin{equation}\label{prj}
\begin{array}{l}
w\,\,\mathrm{\hs\ is\hs\ pro\-ject\-a\-ble\hs\ along\hs\
}\,\,\dz\,\,\mathrm{\hs\ if\hs\ and\hs\ only\
if,\hs\ for\hs\ every\hs\ sec}\hyp\\
\mathrm{tion\ }\,v\,\mathrm{\ of\ }\,\dz\,\mathrm{\ the\ Lie\ 
bracket\ }\,\nh[w,\nh v]\,\mathrm{\ is\ also\ a\ section\ of\ 
}\,\dz.
\end{array}
\end{equation}
One says that a $\,(0,r)\,$ tensor field $\,\xi\,$ on a manifold $\,M\,$ 
{\it annihilates\/} a distribution $\,\dz\,$ (or, is {\it pro\-ject\-a\-ble 
along\/} $\,\dz$) if $\,\xi(v_1\w,\dots,v\hn_r\w)=0\,$ whenever one of the 
vector fields $\,v_1\w,\dots,v\hn_r\w$ is a section of $\,\dz\,$ or, 
respectively, if $\,\dz\,$ is in\-te\-gra\-ble, and $\,\xi$, locally, equals
the pull\-back to $\,M\,$ of a $\,(0,r)\,$ tensor field on a local leaf space
of $\,\dz$.
\begin{lemma}\label{annih}Let\/ $\,\dz\,$ be a $\,\nabla\nnh$-par\-al\-lel 
distribution on a manifold $\,M\,$ with a tor\-sion-free connection
$\,\nabla\nh$. If a $\,(0,r)\,$ tensor field\/ $\,\xi\nh$ on\/ $\,M\hs$ 
annihilates\/ $\,\dz$, then so does\/ $\,\nav\xi$, for any
vector field $\,v$, while\/ $\,\nav\xi=0\,$ when, in addition, $\,\xi\nh$ is
pro\-ject\-a\-ble along\/ $\,\dz$ and\/ $\,v\hs$ is a section of\/ $\,\dz$.
\end{lemma}
We get both claims evaluating 
$\,[\nav\xi](v_1\w,\dots,v\hn_r\w)\,$ from the Leib\-niz rule; 
the second, as $\,v_1\w,\dots,v\hn_r\w$ may be assumed
pro\-ject\-a\-ble 
along $\,\dz$, so that $\hs\xi(v_1\w,\dots,v\hn_r\w)\hs$ is constant along
$\,\dz$, while, by (\ref{prj}), each $\,[v,v_i\w]$, and hence 
$\,\nav\nh v_i\w$, is a section of $\,\dz$.
\begin{lemma}\label{coord}With the index ranges\/ $\,j=1,\dots,r\,$ and\/
$\,k=r+1,\dots,n$, let functions\/ $\,x^k$ on an\/ $\,n$-di\-men\-sion\-al
manifold\/ $\,M\,$
have\/
$\,dx^{r+1}\nh\nnh\wedge\ldots\hs\wedge\,dx^n\nh\ne0\,$ everywhere and be
constant along mutually commuting vector fields\/ $\,e\nh_j\w$ that are
linearly independent at each point. Then, locally, there exist
coordinates\/ $\,x^1\nnh,\dots,x^n$ including our\/ $\,x^k$ for which\/
$\,e\nh_j\w$, $\,j=1,\dots,r$, are the coordinate vector fields\/
$\,\partial\nh_j\w$.
\end{lemma}
\begin{proof}Fix a co\-di\-men\-sion $\,r\hn$ sub\-man\-i\-fold $\,Q\hs$
transverse to the span of our $\,e\nh_j\w$. Let 
$\,F(z,x^1\nh,\dots,x^r)=x(1)\,$ for $\,z\in Q\,$ and real
$\,x^1\nh,\dots,x^r$ such that $\,x(1)$ exists for the integral curve
$\,t\mapsto x(t)\,$ of the combination $\,x^je\nnh_j\w$ with $\,x(0)=z$. 
By the inverse mapping theorem, $\,F\,$ provides a dif\-feo\-mor\-phic
identification of a neighborhood of any given point $\,z\,$ in $\,M\,$
with the Car\-te\-sian product of a neighborhood of $\,z\,$ in $\,Q\,$ and 
a neighborhood of zero in the Euclidean $\,r$-space. 
This turns the variables $\,x^j$ into the required additional coordinate
functions.
\end{proof}
\begin{lemma}\label{plbcx}For a closed differential\/ $\,\px$-form\/  
$\,\eta\,$ on a Car\-te\-sian-prod\-uct manifold $\,M\hs$ with\/ 
$\,T\nh M\nh=\hs\hz^+\nnh\oplus\,\hz^-$ for the factor distributions\/
$\,\hz^\pm\nnh$, 
let $\,\eta=\eta^+\nnh\hn+\hs\eta^-\nnh$, where each of the\/ $\,\px$-forms\/  
$\,\eta^\pm\hn\nnh$ annihilates\/ $\,\hz^\mp\nnh$. Then both\/
$\,\eta^\pm\hn\nnh$
are the pull\-backs to\/ $\,M\hs$ of some closed $\,\px$-forms on the factor 
manifolds.
\end{lemma}
\begin{proof}Applying (\ref{dbr}) to suitable coordinate vector fields for 
a Car\-te\-sian-prod\-uct coordinate system we see that $\,d\eta^\pm\nnh=0\,$ 
and the component functions of $\,\eta^\pm$ are constant along $\,\hz^\mp\nh$,
as required.
\end{proof}
\begin{remark}\label{divrg}Locally in $\,\rp\nh$, every function $\,\phi\,$ is
the divergence of some vector field $\,w$, for instance,
$\,w=(\psi,0,\dots,0)\,$ with $\,\psi\,$ such that 
$\,\hs\partial\nh_1\w\hn\psi=\phi$.
\end{remark}
\begin{remark}\label{dufrm}Given an $\,n$-di\-men\-sion\-al manifold $\,M\,$ 
and a local trivialization $\,\xi^1\nnh,\dots,\hs\xi^n$ of
$\,T^*\hskip-2.1ptM\,$ 
dual to a local trivialization $\,e_1\w,\dots,e_n\w$ of $\,T\nh M\,$ with
functions $\,C_{i\hn j}^k$ such that $\,[\hs e\hn_i\w,e\nh_j\w]=C_{i\hn j}^ke\hn_k\w$,
one has, by (\ref{dbr}) and (\ref{cvt}), 
$\,[\hh d\hh\xi^k](e\hn_i\w,e\nh_j\w)=-C_{i\hn j}^k\nh$, and 
hence $\,d\hh\xi^k\nh=-C_{i\hn j}^k\hs\xi\hh^i\nnh\nh\wedge\hs\xi^j\nh$. 
\end{remark}
\begin{remark}\label{clsed}If $\,\px\ge2$, any $\,(\px,0)\,$ differential form
$\,\omega\,$ on a complex manifold, having a closed real part, must 
itself be closed (that is, hol\-o\-mor\-phic). In fact, 
$\,d\hs\omega=\partial\hs\omega+\,\overline{\hskip-2pt\partial}\hs\omega\,$ is
then imaginary, and hence opposite to its conjugate
$\,d\,\hh\overline{\hskip-1pt\omega\hskip-1pt}$, while
$\,d\,\hh\overline{\hskip-1pt\omega\hskip-1pt}\,$ has
bi\-ho\-mo\-ge\-ne\-ous components of bi\-de\-grees
$\,(0,\px+1)\,$ and $\,(1,\px)$, different of the bi\-de\-grees 
$\,(\px+1,0)\,$ and $\,(\px,1)\,$ for $\,d\hs\omega\,$ unless
$\,\px\in\{0,1\}$.
\end{remark}

\section{Some invariants of exterior forms}\label{ie}
\setcounter{equation}{0}
Throughout this section $\,V\hh$ is a real vector space of dimension
$\,n$. We call a $\,\px\hs$-vec\-tor $\,\beta\in V^{\wedge\px}$ or an
exterior $\,\px\hs$-form $\,\my\in[V\hn^*]^{\wedge\px}$ {\it
de\-com\-pos\-able\/} if it is the exterior product of $\,\px\,$ vectors or
$\,1$-forms. A {\it volume form\/} in $\,V\hs$ is a nonzero 
exterior $\,n$-form, which amounts to a nonzero scalar when $\,n=0$.

Let $\,\my\in[V\hn^*]^{\wedge\px}$ be an exterior $\,\px\hs$-form in 
$\,V\nnh$, where $\,1\le\px\le n$. Its  {\it rank\/} 
is the minimum
dimension of a vector space $\,W\nh$ such that $\,\my\,$ equals the pull\-back
of an exterior $\,\px\hs$-form in $\,W\nh$ under some linear operator
$\,V\to W\nnh$. Since $\,V\to W\hs$ may be assumed surjective, when
$\,r=\mathrm{rank}\hskip2.7pt\my\,$ and $\,\my\ne0$,
\begin{equation}\label{prn}
\px\ne r-1\,\mathrm{\ and\
}\px\le r\hn\le n\,\mathrm{\ with\ }\,\px=r\,\mathrm{\ if\ and\ only\ if\
}\,\my\,\mathrm{\ is\ de\-com\-pos\-able,}
\end{equation}
due to the well-known de\-com\-pos\-ab\-il\-i\-ty of $\,r\hn$-forms and
$\,(r\hn-1)$-forms in dimension $\,r$. See, for instance,
\cite[pp.\,287-288]{kostrikin-manin}, 
\cite[Examples\,1.6,\,1.8]{munoz-masque-pozo-coronado-rosado-maria} or 
\cite[Sect.\,11]{derdzinski-piccione-terek}.

We associate with $\,\my\,$ two vector sub\-spaces of $\,V\nnh$. One is the
{\it kernel\/} $\,Z\,$ of $\,\my$, in other words, the kernel of the operator
$\,V\nh\to[V\hn^*]^{\wedge(\px-1)}$ sending $\,v\,$ to
$\,\my(v,\,\cdot\,,\dots,\,\cdot\,)$. 
The other space, which we call the {\it divisibility space\/} of $\,\my\,$
and denote by $\,D$, is the polar space (annihilator) of the sub\-space
$\,D'$ of $\,[V\hn^*]\,$ consisting
of $\,1$-forms $\,\xi\in V\hn^*$ such that $\,\my\,$ is
$\,\wedge$-di\-vis\-i\-ble by $\,\xi\,$ (or, equivalently,
$\,\xi\wedge\my=0$). Thus, $\,D\,$ is the
simultaneous kernel of all such $\,1$-forms $\,\xi$. Then, for
$\,r=\mathrm{rank}\hskip2.7pt\my\,$ and $\,k=\dim Z$,
\begin{equation}\label{rnk}
\mathrm{a)}\hskip6ptk\,=\,n\,-\,r,\qquad\mathrm{b)}\hskip6ptZ\,
\subseteq\hs D\,\mathrm{\ unless\ }\,\my=0.
\end{equation}
In fact, (\ref{rnk}-a) follows since $\,\my\,$ clearly equals
the pull\-back under the projection operator $\,V\to V\nnh/Z\,$ of the exterior
$\,\px\hs$-form in $\,V\nnh/Z\,$ that $\,\my\,$ descends to, the
min\-i\-mum-di\-men\-sion clause being obvious as a pull\-back form vanishes
on the kernel of the operator used to pull it back. To obtain (\ref{rnk}-b),
note that, if $\,v\in Z\,$ and $\,\xi\in D'\nh$, one has
$\,\xi(v)\my(v_1\w,\dots,v\hn_\px\w)=[\hs\xi\wedge\my](v,v_1\w,\dots,v\hn_\px\w)=0\,$
for all $\,v_1\w,\dots,v\hn_\px\w\in V\nh$.

When $\,\my=0\,$ and $\,V\nh\ne\{0\}$, (\ref{rnk}-b) fails to
hold: $\,Z\nh=\hh V\hh$ and $\,D=\{0\}$. For nonzero scalars ($0$-forms)
$\,\my\,$ we set $\,Z\nh=\{0\}\,$ and $\,D=V\nnh$. Generally,
\begin{equation}\label{ind}
\mathrm{we\ call\ an\ exterior\ form\ }\,\my\,\text{\ \it
in\-di\-vis\-i\-ble}\mathrm{\hs\ if\ }\,D=V\nnh,
\end{equation}
that is, if $\,\xi\wedge\my\ne0\,$ whenever
$\,\xi\in V\hn^*\nh\smallsetminus\{0\}$. We will repeatedly assume that
\begin{equation}\label{bas}
\xi^1\nnh,\dots,\hs\xi^n\mathrm{\ is\ the\ basis\ of\ }\,V\hn^*\nnh\mathrm{\
dual\ to\ a\ basis\ }\,e_1\w,\dots,e_n\w\mathrm{\ of\ }\,V\nh. 
\end{equation}
Then, by (\ref{cvt}), for
$\,\xi=\my(\,\cdot\,,e\nh_{j_2}\w,\dots,e\nh_{j_\px}\w)\,$ with 
$\,\my=\xi^{i_1}\nnh\wedge\ldots\wedge\hs\xi^{i_\px}\nnh\ne0$,
\begin{equation}\label{wdg}
\begin{array}{l}
\xi\hs=\hs\pm\hs\xi\hh^i\,\mathrm{\ if\ 
}\,\{i_1\w,\dots,i_\px\w\}\,
=\,\{i\}\hs\cup\{j_2\w,\dots, j_\px\w\}\mathrm{,\hh\
and\ }\,\hs\xi=\hs0\\
\mathrm{when\ 
}\,\{i_1\w,\dots,i_\px\w\}
\smallsetminus\{j_2\w,\dots, j_\px\w\}\,\mathrm{\
is\ not\ a\ one}\hyp\mathrm{el\-e\-ment\ set.}
\end{array}
\end{equation}
\begin{lemma}\label{divpr}If\/ 
$\,\xi^1\nnh,\dots,\hs\xi^s\nh\in V\hn^*$ are linearly independent, 
any exterior form\/ $\,\eta$ with\/ 
$\,\xi^1\nnh\nh\wedge\ldots\wedge\hs\xi^s\nnh\nh\wedge\eta=0\,$ 
lies in the ideal generated by\/ 
$\,\xi^1\nnh,\dots,\hs\xi^s\nnh$.
\end{lemma}
\begin{proof}Let $\,\eta\in[V\hn^*]^{\wedge\px}\nnh$. Expanding $\,\eta\,$ 
as a linear combination of the obvious basis of $\,[V\hn^*]^{\wedge\px}$
arising from a basis $\,\xi^1\nnh,\dots,\hs\xi^n$ of $\,V\hn^*$ which includes 
our $\,\xi^1\nnh,\dots,\hs\xi^s\nh$, we see that the $\,\px\hs$-fold 
exterior products without any of the factors
$\,\xi^1\nnh,\dots,\hs\xi^s$ occurring in the expansion of
$\,\eta\,$ with nonzero coefficients would remain linearly independent even
after being $\,\wedge$-mul\-ti\-plied by
$\,\xi^1\nnh\nh\wedge\ldots\wedge\hs\xi^s\nh$. As
$\,\xi^1\nnh\nh\wedge\ldots\wedge\hs\xi^s\nnh\nh\wedge\eta=0$, there are no
$\,\px\hs$-fold products with the above properties.
\end{proof}
\begin{remark}\label{image}The {\it image\/} of
$\,\my\in[V\hn^*]^{\wedge\px}\nnh$, defined to be the span in $\,V\hn^*$ of 
all $\,\my(\,\cdot\,,v_2\w,\dots,v\hn_\px\w)\,$ for
$\,v_2\w,\dots,v\hn_\px\w\in V\nnh$. Obviously,
\begin{equation}\label{dim}
\begin{array}{l}
\mathrm{the\ image\ of\ }\,\my\,\mathrm{\ is\ the\ polar\ space\ of\ 
}\,\mathrm{Ker}\,\my\mathrm{,\ so\ that,\ by}\\
\mathrm{(\ref{rnk}}\hyp\mathrm{a),\ the\ dimension\ of\ the\ image\ of\
}\,\my\,\mathrm{\ equals\ 
}\,\hs\mathrm{rank}\hskip2.7pt\my
\end{array}
\end{equation}
Assuming (\ref{bas}), we easily see that 
\begin{equation}\label{ims}
\mathrm{the\ image\ of\ }\,\my\,\mathrm{\ is\ spanned\ by\
}\hs\{\my(\,\cdot\,,e\nh_{j_2}\w,\dots,e\nh_{j_\px}\w)\nh:\nh
1\le j_2\w<\ldots<j_\px\w\le n\}
\end{equation}
and so, as a consequence of (\ref{wdg}), the image of $\,\my\,$ is then
\begin{equation}\label{ctd}
\mathrm{contained\ in\ the\ span\ of\ all\ }\,\xi\hh^i\mathrm{\ occurring\
in\ }\,\my\hh.
\end{equation}
The word `occurring\nh' means here that the expansion of $\,\my\,$ 
as a linear combination of the obvious basis of $\,[V\hn^*]^{\wedge\px}$ 
includes, with a nonzero coefficient, a $\,\px\hs$-fold 
exterior product 
involving the factor $\,\xi\hh^i\nh$.
\end{remark}
\begin{remark}\label{imges}Under the assumption (\ref{bas}), 
we will say that a $\,\px\hs$-el\-e\-ment
set $\,I\subseteq\{1,\dots,n\}\,$ {\it occurs in\/} an exterior
$\,\px\hs$-form $\,\my\,$ if the expansion of $\,\my\,$
as a linear combination of the obvious basis of $\,[V\hn^*]^{\wedge\px}$
includes $\,\xi^{i_1}\nnh\wedge\ldots\wedge\hs\xi^{i_\px}$ with a nonzero
coefficient, where $\,I\nh=\{i_1\w,\dots,i_\px\w\}$. Let $\,S\,$ be the union
of the distinct $\,\px\hs$-el\-e\-ment sub\-sets $\,I\nnh_1\w,\dots,I\nh_l\w$
of $\,\{1,\dots,n\}\,$ occuring in $\,\my$. If $\,\px\ge2\,$ and 
any $\,(\px-1)$-el\-e\-ment sub\-set of $\,\{1,\dots,n\}\,$ 
is contained in at most one of $\,I\nnh_1\w,\dots,I\nh_l\w$ (for instance,
$\,I\nnh_1\w,\dots,I\nh_l\w$ are pairwise disjoint), 
then the image of $\,\my$ is the span of $\,\{\xi\hh^i:i\in S\}\,$ (and hence, 
by (\ref{dim}), $\,\mathrm{rank}\hskip2.7pt\my=|S|$, so that 
$\,\mathrm{rank}\hskip2.7pt\my=\px\,$ if $\,\my\,$ is de\-com\-pos\-able,
with $\,l=1$). In fact, one inclusion is provided by (\ref{ctd}). For the
other one, we fix $\,\xi\hh^i$ with $\,i\in S$ and apply (\ref{wdg}) to 
$\,j_2\w,\dots, j_\px\w$ such that
$\,\{i\}\cup\{j_2\w,\dots, j_\px\w\}\,$ is one of $\,I\nnh_1\w,\dots,I\nh_l\w$.
\end{remark}
\begin{remark}\label{dvker}The kernel of any nonzero de\-com\-pos\-able
$\,\px\hs$-form $\,\my\,$ coincides with its divisibility space: writing
$\,\my=\xi^1\nnh\nh\wedge\ldots\wedge\hs\xi\hh^\px\nnh$, with (\ref{bas}), we
see that both have the same polar space 
$\,\mathrm{Span}\hs(\xi^1\nnh,\dots,\hs\xi\hh^\px)$. (The former according to
(\ref{dim}) and Remark~\ref{imges}, the latter since the equality 
$\,\xi^1\nnh\nh\wedge\ldots\wedge\hs\xi\hh^\px\nnh\nh\wedge\hs\xi=0\,$ for a 
$\,1$-form $\,\xi\hh$ amounts to linear dependence of the system 
$\,\xi^1\nnh,\dots,\hs\xi\hh^\px\nnh,\hs\xi$.)
\end{remark}
\begin{remark}\label{nondc}For $\,\px\ge2\,$ and linearly independent
$\,1$-forms $\,\xi^1\nnh,\dots,\hs\xi\hh^{\px+2}\nh\in V\hn^*\nh$,
the $\,\px\hs$-form $\,\my=(\xi^1\nnh\nh\wedge\hs\xi^2\hn
+\hs\xi^3\nnh\nh\wedge\hs\xi\hs^4)\hs\wedge
\xi\hh^5\nh\nnh\wedge\ldots\wedge\hs\xi\hh^{\px+2}$ is not de\-com\-pos\-able:
Remark~\ref{imges} gives $\,\mathrm{rank}\hskip2.7pt\my=\px+2\,$ for our
$\,\my$, and rank equal to $\,\px\,$ for de\-com\-pos\-able $\,\px\hs$-forms.
\end{remark}
\begin{lemma}\label{dirpd}The divisibility space of an exterior\/
$\,\px\hs$-form\/ 
$\,\my=\theta\wedge\zeta\,$ is\/ $\,\{0\}\times D$ whenever\/
$\,\theta\,$ and\/ $\,\zeta$ are the pull\-backs to the di\-rect-prod\-uct
vector space\/ $\,V\nh=W\nnh\times\hn D$ of a volume form in\/ 
$\,W\hs$ and an in\-di\-vis\-i\-ble exterior form in\/ $\,D$.
\end{lemma}
\begin{proof}For $\,e_1\w,\dots,e_n\w$ with (\ref{bas}) having the first
$\,s\,$ vectors in
$\,W\nh\times\{0\}\,$ and the last $\,n-s\,$ in $\,\{0\}\times D$, such that 
$\,\theta=\xi^1\nnh\nh\wedge\ldots\wedge\hs\xi^s\nnh$, expanding
$\,\xi\wedge\my\,$ as a linear combination of the
obvious basis of $\,[V\hn^*]^{\wedge\px}\nnh$, where
$\,\xi=a\hn_i\w\hs\xi\hh^i\nh$, and 
``canceling'' the factor $\,\xi^1\nnh\nh\wedge\ldots\wedge\hs\xi^s$ in each
nonzero term of the expansion, we see that $\,\xi\wedge\my=0\,$ if and
only if $\,\zeta\,$ is $\,\wedge$-di\-vis\-i\-ble by 
$\,a\hn_{s+1}\w\hs\xi\hh^{s+1}\nnh+\ldots+a\hn_n\w\hs\xi\hh^n$ when viewed
as a form in $\,D$, which amounts to $\,a\hn_{s+1}\w=\ldots=a\hn_n\w=0$.
\end{proof}

\section{Further invariants}\label{fi}
\setcounter{equation}{0}
As before, $\,V\hh$ is a real vector space of dimension $\,n$. 
The following lemma may be thought of as a converse of Lemma~\ref{dirpd}.
\begin{lemma}\label{divis}Let\/ $\,\my\in[V\hn^*]^{\wedge\px}\nh$ be a nonzero
exterior\/ $\,\px\hs$-form of rank\/ $\,r\hs$ in\/
$\,V\nnh$,
with\/ $\,k=n-\hs r\hs=\,\dim Z\,$ and\/ $\,s=n-\dim D\hs\le\,r\hs$ for the
kernel\/ $\,Z\,$ and divisibility space\/ $\,D\,$ of\/ $\,\my$, cf.\
{\rm(\ref{rnk})}. This has four consequences.
\begin{enumerate}
\item[{\rm(a)}] $s\ne\px-1\,$ and\/ $\hs\,s\le\px$, with equality if and only
if\/ $\,\my\,$ is de\-com\-pos\-able.
\item[{\rm(b)}] For any basis\/ 
$\,\xi^1\nnh,\dots,\hs\xi^n$ of\/ $\,V\hn^*$ such that\/ $\,\xi^1\nnh,\dots,\hs\xi^s$
is a basis of\/ $\,D'\nnh$, the polar space of\/ $\,D$, one has\/ 
$\,\my=\xi^1\nnh\nh\wedge\ldots\wedge\hs\xi^s\nnh\wedge\zeta$, where the
in\-di\-vis\-i\-ble exterior\/ $\,(\px-\nh s)$-form\/ 
$\,\zeta\hs$ is a linear combination
of\/ $\,(\px-\nh s)$-factor exterior products of\/ $1$-forms from the set
$\,\{\xi^{s+1}\nnh,\dots,\hs\xi^n\}$.
\item[{\rm(c)}] If\/
$\,\my=\xi^1\nnh\nh\wedge\ldots\wedge\hs\xi^s\nnh\wedge\zeta\,$ for some
basis\/
$\,\xi^1\nnh,\dots,\hs\xi^s$ of\/ $\,D'$ and some exterior\/
$\,(\px-\nh s)$-form\/ 
$\,\zeta$, then the restriction of any such\/ $\,\zeta\hs$ to\/ $\hs\,D\,$ 
is, uniquely, up to a nonzero scalar factor, determined by\/ $\,\my$.
\item[{\rm(d)}] The above restriction of\/
$\,\zeta\hs$ to\/ $\hs\,D\,$ is in\-di\-vis\-i\-ble in\/ $\,D$.
\end{enumerate}
\end{lemma}
\begin{proof}Given a basis $\,\xi^1\nnh,\dots,\hs\xi^n$ as in (b), 
$\,\my\,$ is a nonzero\hh-co\-ef\-fi\-cients linear combination of several
exterior products $\,\xi^{i_1}\nnh\wedge\ldots\wedge\hs\xi^{i_q}$ with
$\,i_1\w\nh<\ldots<i_q\w$. The equalities $\,\xi\hh^i\nnh\nh\wedge\my=0\,$
for $\,i=1,\dots,s\,$ amount to 
$\,\{1,\dots,s\}\subseteq\{i_1\w\nh,\ldots,i_q\w\}\,$ for each
$\,\xi^{i_1}\nnh\wedge\ldots\wedge\hs\xi^{i_q}$ present in the
combination, leading to the required decomposition
$\,\my=\xi^1\nnh\nh\wedge\ldots\wedge\hs\xi^s\nnh\wedge\zeta$, where
$\,\zeta\,$ arises by ``canceling'' the factor
$\,\xi^1\nnh\nh\wedge\ldots\wedge\hs\xi^s$ in each (nonzero) term of our
expansion of $\,\my$, so that $\,\zeta\,$ is ``built'' from 
$\,\xi^{s+1}\nnh,\dots,\hs\xi^n\nh$. Any $\,1$-form $\,\xi$, 
$\,\wedge$-di\-vid\-ing $\,\zeta$, also divides $\,\my$, and so
$\,\xi=a_1\w\xi^1\nh+\ldots+a_s\w\xi^s$ for some $\,a_1\w,\dots,a_s\w$. 
Since $\,\xi^1\nnh,\dots,\hs\xi^s$ are not present in $\,\zeta$, writing
$\,\xi\wedge\zeta=0\,$ we see that $\,a_1\w=\dots=a_s\w=0$, and (b) follows. 
So does (a): if our $\,\zeta\hs$ 
were {\it a\/ $\,1$\hn-form}, obviously $\,\wedge$-di\-vid\-ing $\,\my$, it
would have to be a linear combination of $\,\xi^1\nnh,\dots,\hs\xi^s$ rather
than being built from $\,\xi^{s+1}\nnh,\dots,\hs\xi^n\nh$.

Next, let $\,\my=\xi^1\nnh\nh\wedge\ldots\wedge\hs\xi^s\nnh\wedge\zeta
=\xi^1\nnh\nh\wedge\ldots\wedge\hs\xi^s\nnh\wedge\zeta'\nh$, as in (c).
Lemma~\ref{divpr} for $\,\eta=\zeta-\zeta'\nh$, combined with uniqueness 
of $\,\xi^1\nnh\nh\wedge\ldots\wedge\hs\xi^s$ up to a factor, yields (c).

For (d), choose $\,\zeta\,$ as in (b). The restrictions of 
$\,\xi^{s+1}\nnh,\dots,\hs\xi^n$ to $\,D\,$ form a basis of $\,D^*\nh$. 
Any linear combination of these restrictions, $\,\wedge$-di\-vid\-ing 
$\,\zeta$, thus $\,\wedge$-di\-vides $\,\my$, which makes it 
also a linear combination of $\,\xi^1\nnh,\dots,\hs\xi^s$, and hence zero.
\end{proof}
We will refer to the restriction to $\,D\,$ of the exterior
$\,(\px-\nh s)$-form $\,\zeta\,$ in part (c) of Lemma~\ref{divis} as an {\it
in\-di\-vis\-i\-ble factor\/} 
of the exterior $\,\px\hs$-form $\,\my\,$ in $\,V\nnh$, and to 
$\,\theta=\xi^1\nnh\nh\wedge\ldots\wedge\hs\xi^s$ in (c) as a {\it volume
factor\/} of $\,\my$. Since $\,(n-\nh s)-(\px-\nh s)=n-\px$,
\begin{equation}\label{cod}
\mathrm{an\ in\-di\-vis\-i\-ble\ factor\ of\ 
}\,\my\,\mathrm{\ has\ the\ same\ codegree\ in\ }\,D\,\mathrm{\ as\ 
}\,\my\,\mathrm{\ does\ in\ }\,V\nnh.
\end{equation}
Also, $\,\xi^1\nnh,\dots,\xi^s$ descend to a basis 
of $\,V\nnh/D$, and so
\begin{equation}\label{vol}
\theta\hs=\,\xi^1\nnh\nh\wedge\ldots\wedge\hs\xi^s\mathrm{\ descends\ to\ a\ volume\
form\ in\ }\,V\nnh/D\hh.
\end{equation}
When $\,s=\px\,$ (which is the de\-com\-pos\-able case in
Lemma~\ref{divis}(a)) we can obviously make $\,\zeta\,$ unique, by
setting $\,\zeta=1$.
\begin{remark}\label{replc}If $\,n=6\,$ and $\,\xi^1\nnh,\dots,\hs\xi\hh^6$
is a basis of $\,V\hn^*\nnh\nnh$, consider the equality
\begin{equation}\label{eql}
\begin{array}{l}
\xi^1\nnh\nh\wedge\hs\xi^2\hn\nnh\wedge\hs\xi^3\hs
+\,\hs\xi^3\nnh\nh\wedge\hs\xi\hs^4\nh\nnh\wedge\hs\xi\hh^5\hs
+\,\hs\xi\hh^5\nnh\nh\wedge\hs\xi\hh^6\hn\nnh\wedge\hs\xi^1\\
\hskip50.3pt=\,\hs\hat\xi^1\nnh\nh\wedge\hs\hat\xi^2\hn\nnh\wedge\hs\hat\xi^3\hs
+\,\hs\hat\xi^3\nnh\nh\wedge\hs\hat\xi\hs^4\nh\nnh\wedge\hs\hat\xi\hh^5\hs
+\,\hs\hat\xi\hh^5\nnh\nh\wedge\hs\hat\xi\hh^6\hn\nnh\wedge\hs\hat\xi^1
\end{array}
\end{equation}
for some $\,\hat\xi^1\nnh,\dots,\hs\hat\xi\hh^6\nh\in V\hn^*$ (which must
then form a basis since, according to Remark~\ref{imges}, the $\,3$-form on 
the left-hand side has rank six, while linear dependence
of $\,\hat\xi^1\nh,\dots,\hat\xi\hh^6$ would make the rank of the right-hand
side less than six, due to the original definition of rank at the beginning 
of this section). In such $\,\hat\xi^1\nnh,\dots,\hs\hat\xi\hh^6\nnh$,
\begin{enumerate}
\item[{\rm(a)}] $\hat\xi^1\nh,\hat\xi^3\nh,\hat\xi\hh^5$ can be any triple
with $\,\mathrm{Span}\hs(\hat\xi^1\nh,\hat\xi^3\nh,\hat\xi\hh^5)
=\hs\mathrm{Span}\hs(\xi^1\nnh,\hs\xi^3\nnh,\hs\xi\hh^5)$, or
\item[{\rm(b)}] $\hat\xi^2$ may be any nonzero $\,1$-form in 
$\,\hs\mathrm{Span}\hs(\xi^2\nnh,\hs\xi\hs^4\nnh,\hs\xi\hh^6)$.
\end{enumerate}
In fact, we get (a) by substituting for 
$\,\xi^1\nnh,\hs\xi^3\nnh,\hs\xi\hh^5$ in (\ref{eql}) arbitrary 
linearly independent linear combinations of 
$\,\hat\xi^1\nh,\hat\xi^3\nh,\hat\xi\hh^5$ and gathering terms which have 
the form $\,\hs\hat\xi^j\nnh\nh\wedge\ldots\wedge\hat\xi^k$ for 
$\,(j,k)\,$ equal to $\,(1,3)$, $\,(3,5)\,$ and $\,(5,1)$. To obtain (b),
note that the above process replaces $\,\xi^2$ with
$\,\hat\xi^2\nh=c_5\w\xi^2\nh+c_1\w\xi\hs^4\nh+c_3\w\xi\hh^6\nh$, where 
$\,(c_1\w,c_3\w,c_5\w)\in\rtr$ is the vector product of the first two rows 
of the $\,3\times3\,$ matrix $\,B\,$ satisfying the matrix equality
$\,[\hs\xi^1\ \ \xi^3\ \ \xi\hh^5]
=[\hs\hat\xi^1\ \ \hat\xi^3\ \ \hat\xi\hh^5]B$. Any prescribed nonzero
vector product is realized in this way by some nonsingular $\,3\times3\,$
matrix. 
\end{remark}
\begin{lemma}\label{prddc}Given an exterior\/ $\,3$-form\/ $\,\my\,$ in a 
six-di\-men\-sion\-al real vector space\/ $\,V\hs$ and a basis\/
$\,\xi^1\nnh,\dots,\hs\xi\hh^6$ of $\,V\hn^*\nnh$, let\/ $\,H'$ be the set of
all\/ $\,1$-forms\/ $\,\xi\in V\hn^*$ such that the\/ $\,4$-form\/
$\,\xi\wedge\my\,$ is de\-com\-pos\-able.
\begin{enumerate}
\item[(i)] If\/ $\,\my\hs
=\hs\xi^1\nnh\nh\wedge\hs\xi^2\hn\nnh\wedge\hs\xi^3\nh
+\hs\xi^3\nnh\nh\wedge\hs\xi\hs^4\nh\nnh\wedge\hs\xi\hh^5\nh
+\hs\xi\hh^5\nnh\nh\wedge\hs\xi\hh^6\hn\nnh\wedge\hs\xi^1\nnh$,
$\,H'$ equals the vector sub\-space of\/ $\,V\hn^*$ spanned by\/
$\,\xi^1\nnh,\hs\xi^3\nnh,\hs\xi\hh^5\nnh$.
\item[(ii)] When\/ $\,\my\hs
=\hs\xi^1\nnh\nh\wedge\hs\xi^2\hn\nnh\wedge\hs\xi^3\nh
+\hs\xi\hs^4\nnh\nh\wedge\hs\xi\hh^5\nh\nnh\wedge\hs\xi\hh^6\nnh$,
our\/ $\,H'$ is the set-the\-o\-ret\-i\-cal union of two vector 
sub\-spaces of\/ $\,V\hn^*\nnh$, the spans of\/
$\,\xi^1\nnh,\hs\xi^2\nnh,\hs\xi^3$ and\/ 
$\,\xi\hs^4\nnh,\hs\xi\hh^5\nnh,\hs\xi\hh^6\nnh$.
\end{enumerate}
\end{lemma}
\begin{proof}In (i), or (ii), 
$\,\xi\wedge\my\,$ is de\-com\-pos\-able, for a $\,1$-form 
$\,\xi=a\hn_i\w\hs\xi\hh^i\nh$, if and only if
$\,(a_2\w,a_4\w,a_6\w)=(0,0,0)\,$ or, respectively, 
one of $\,(a_1\w,a_2\w,a_3\w)$, $\,(a_4\w,a_5\w,a_6\w)\,$ equals $\,(0,0,0)$.
Namely, the `if' part is 
easily verified in both cases.

For the converse, in (i), let 
$\,(a_2\w,a_4\w,a_6\w)\ne(0,0,0)$. Remark~\ref{replc}(b) allows us to assume
that $\,\xi=\hs\xi^2\nh
+a\hn_1\w\hs\xi^1\nh+a\hn_3\w\hs\xi^3\nh+a\hn_5\w\hs\xi\hh^5\nnh$, and so 
\[
\xi\wedge\my=\hs\xi^2\nh\nnh\wedge\xi\hh^5\nnh\nh
\wedge(\xi^3\nnh\nh\wedge\hs\xi\hs^4\nh+\hs\xi\hh^6\nh\nnh\wedge\hs\xi^1)
+(a\hn_1\w\hs\xi\hs^4\nh+a\hn_3\w\hs\xi\hh^6\nh+a\hn_5\w\hs\xi^2)
\wedge\hs\xi^1\nnh\nh\wedge\hs\xi^3\hn\nnh\wedge\hs\xi\hh^5\nnh.
\]
In terms of the basis $\,e_1\w,\dots,e_6\w$ of $\,V\hs$ dual to our 
basis $\,\xi^1\nnh,\dots,\hs\xi\hh^6$ of $\,V\hn^*\nnh$, if we now set 
$\,\eta_{i\hn jk}\w=[\hs\xi\wedge\my](\,\cdot\,,e\nh_i\w,e\nnh_j\w,e\nh_k\w)$, 
(\ref{wdg}) will give  
$\,\eta\hh_{234}\w=\xi\hh^5\nnh$, $\,\eta\hh_{245}\w=\xi^3$ and 
$\,\eta\hh_{256}\w=\xi^1\nnh$, as well as
$\,\eta\hh_{354}\w=\xi^2\nh+a_1\w\hs\xi^1$ and 
$\,\eta\hh_{253}\w=\xi\hs^4\nh+a_5\w\hs\xi^1\nnh$. Thus, by (\ref{dim}),
$\,\mathrm{rank}\hskip2.7pt[\hs\xi\wedge\my]\ge5$ and 
$\,\hs\xi\wedge\my\,$ is not de\-com\-pos\-able: if it were, 
it would have rank four (Remark~\ref{imges}).

In (ii), let 
$\,\hat\xi^1\nh=a\hn_1\w\hs\xi^1\nh+a\hn_2\w\hs\xi^2\nh
+a\hn_3\w\hs\xi^3$ and $\,\hat\xi\hs^4\nh=a\hn_4\w\hs\xi\hs^4\nh
+a\hn_5\w\hs\xi\hh^5\nh+a\hn_6\w\hs\xi\hh^6$ be both nonzero, and choose
$\,\hat\xi\hh^i$ for $\,i\in\{2,3,5,6\}\,$ with
$\,\hs\xi^1\nnh\nh\wedge\hs\xi^2\hn\nnh\wedge\hs\xi^3\nh
=\hs\hat\xi^1\nnh\nh\wedge\hat\xi^2\hn\nnh\wedge\hat\xi^3$ and 
$\,\hs\xi\hs^4\nnh\nh\wedge\hs\xi\hh^5\nh\nnh\wedge\hs\xi\hh^6\nh
=\hs\hat\xi\hs^4\nnh\nh\wedge\hat\xi\hh^5\hn\nnh\wedge\hat\xi\hh^6\nh$.
Then  
$\,\xi\wedge\my=(\hat\xi^1\nnh+\hh\hat\xi\hs^4)\wedge\my
=\hs\hat\xi^1\nnh\nh\wedge\hat\xi\hs^4\nh
\wedge(\hat\xi\hh^5\nh\nnh\wedge\hat\xi\hh^6\hn
-\hs\hat\xi^2\hn\nnh\wedge\hat\xi^3)$ 
is not de\-com\-pos\-able as a consequence of Remark~\ref{nondc}.
\end{proof}

\section{The Hodge star duality}\label{hs}
\setcounter{equation}{0}
Again, $\,V\hh$ denotes a real vector space of dimension $\,n$.

We use the multiplicative notation $\,\my\beta=\beta\my\,$ for the natural
bi\-lin\-e\-ar pairing which associates with
$\,\my\in[V\hn^*]^{\wedge\px}$ and $\,\beta\in V^{\wedge\lx}$ the result of
contractions of $\,\my$ against $\,\beta\,$ involving the maximum possible
number of initial indices. Thus, $\,\my\beta\,$ is a $\,(\px-\lx)$-form if
$\,\px\ge\lx$, an $\,(\lx-\px)$-vec\-tor when $\,\lx\ge\px$, and hence a
scalar in the case $\,\lx=\px$, and, when 
$\,\beta=v_1\w\nh\wedge\ldots\wedge v_\lx\w$ and $\,\px\ge\lx\,$ (or, 
$\,\my=\xi^1\nnh\nh\wedge\ldots\wedge\hs\xi\hh^\px$ and $\,\lx\ge\px$), 
$\,\my\beta=\my(v_1\w,\dots,v_\lx\w,\,\cdot\,,\dots,\,\cdot\,)\,$ 
or, respectively,
$\,\my\beta=\beta\my
=\beta(\xi^1\nnh,\dots,\xi\hh^\px,\,\cdot\,,\dots,\,\cdot\,)$.
Note that $\,\my v\,$ is the interior product 
$\,\imath_v\w\my=\my(v,\,\cdot\,,\dots,\,\cdot\,)\,$ for $\,\px\ge1\,$ and
$\,v\in V\nnh$.

For an $\,\lx$-vec\-tor $\,\beta$, a $\,\px\hs$-form $\,\my\,$ and a  
$\,\px'\nh$-form $\,\my'$ one has
the associative law
\begin{equation}\label{law}
[\beta\my]\my'\hs=\,\beta[\my\wedge\my']\,\mathrm{\ if\
}\,\px+\px'\nh\le\lx\hh.
\end{equation}
which are obvious due to (\ref{cvt}) under the assumption that
\begin{equation}\label{dua}
\mathrm{the\ bases\ }\,v_1\w,\dots,v_n\w\mathrm{,\ of\
}\,V\hs\mathrm{\ and\ 
}\,\xi^1\nnh,\dots,\hs\xi^n\mathrm{\ of\ }\,V\hn^*\mathrm{\ are\ each\
other}\text{\rm'}\mathrm{s\ duals,}
\end{equation}
and $\,\beta=v_1\w\nh\wedge\ldots\wedge v_\lx\w$, while 
$\,\my=\xi^1\nnh\nh\wedge\ldots\wedge\hs\xi\hh^\px$ and
$\,\my'\nh=\xi\hh^{\px+1}\nnh\nh\wedge\ldots\wedge\hs\xi\hh^{\px+\px'}\nnh$.
By skew-sym\-me\-try, (\ref{law}) thus follows for $\,\beta,\my,\my'$ that are
all de\-com\-pos\-a\-ble into factors from the bases (\ref{dua}), and
the general case is immediate from tri\-lin\-e\-ar\-i\-ty.

We now fix a volume $\,n$-form $\,\omega\,$ in $\,V$ and the corresponding
reciprocal $\,n$-vec\-tor $\,\alpha$, with $\,\alpha\hh\omega=1$. By
(\ref{cvt}), in the case (\ref{dua}),
if $\,\alpha=v_1\w\nh\wedge\ldots\wedge v_n\w$, then 
$\,\omega=\xi^1\nnh\nh\wedge\ldots\wedge\hs\xi^n\nh$. 
The {\it Hodge star
iso\-mor\-phisms\/} $\,\my\mapsto*\my\,$ and $\,\beta\mapsto*\beta\,$ of 
$\,[V\hn^*]^{\wedge\px}$ onto $\,V^{\wedge(n-\px)}\nnh$, and vice versa,
depending on $\,\omega$, are then given by $\,*\my=\alpha\my\,$ and
$\,*\beta=\omega\beta$. By (\ref{law}),
\[
*[\xi^1\nnh\nh\wedge\ldots\wedge\hs\xi\hh^\px]
=[v_1\w\nh\wedge\ldots\wedge v_n\w]
[\hh\xi^1\nnh\nh\wedge\ldots\wedge\hs\xi\hh^\px]
=v_{\px+1}\w\nh\wedge\ldots\wedge v_n\w
\]
with (\ref{dua}): 
both sides agree on any $\,(n-\px)$-tuple 
$\,\xi\hh^{i_{\px+1}}\nnh\nh\wedge\ldots\wedge\hs\xi\hh^{i_n}$ with 
$\,i_{\px+1}\w\nh<\ldots<i_n\w$ (as they equal $\,1\,$ for
$\,(i_{\px+1}\w,\dots,i_n\w)=(\px+1,\dots,n)\,$ and $\,0\,$ otherwise). 
Evenly permuting the $\,\xi\hh^i$ and,
simultaneously, the $\,v_j\w$, we now get, in the case (\ref{dua}),
\begin{equation}\label{hds}
*[\hh\xi^{i_1}\nnh\wedge\ldots\wedge\hs\xi^{i_\px}]
=v_{i_{\px+1}}\w\nh\wedge\ldots\wedge v_{i_n}\w\hh,\quad
*[v_{i_1}\w\nh\wedge\ldots\wedge v_{i_\lx}\w]
=\xi^{i_{\lx+1}}\nnh\nh\wedge\ldots\wedge\hs\xi^{i_n}
\end{equation}
whenever $\,i_1\w,\dots,i_n\w$ is an even permutation of $\,1,\dots,n$, the
second equality immediate from the first when one switches the roles of
$\,V\hs$ and $\,V\hn^*\nnh$. Hence 
\begin{equation}\label{inv}
\begin{array}{l}
*:V^{\wedge(n-\px)}\nh\to [V\hn^*]^{\wedge\px}\mathrm{\hs\ \ equals\ \ 
}\,(-\nh1)^{(n-\px)\px}\\
\mathrm{times\ the\ inverse\ of\
}*:[V\hn^*]^{\wedge\px}\nh\to V^{\wedge(n-\px)}\nh.
\end{array}
\end{equation}
With the volume $\,n$-form $\,\omega\,$ still fixed, let $\,\beta=*\my$. Then
\begin{equation}\label{img}
\begin{array}{l}
\mathrm{the\ image\ of\ }\hs\beta\hs\mathrm{\ is\ the\ divisibility\ space\
of\ }\hs\my\mathrm{,\nh\ and}\\
\mathrm{the\ divisibility\ space\ of\ }\hs\beta\hs\mathrm{\ equals\ the\ kernel\ of\ }\,\my,
\end{array}
\end{equation}
the two spaces associated with an $\,\lx$-vec\-tor $\,\beta\,$ being defined
in the obvious way:
$\,\{\beta\my'\nh:\my'\in[V\hn^*]^{\wedge(\lx-1)}\}\,$
and $\,\{v\in V:v\wedge\beta=0\}$. In fact, obviously,
\begin{equation}\label{imk}
\mathrm{the\ image\ of\ }\,\beta\,\mathrm{\ is\ polar\ to\ its\ kernel\
}\,\{\xi\in V\hn^*:\beta(\xi,\,\cdot\,,\dots,\,\cdot\,)=0\}\hh.
\end{equation}
Now (\ref{law}) for $\,\alpha\,$ rather than $\,\beta\,$ and
$\,\px'\nh=1\,$ yields the first line of (\ref{img}) by showing that the
two spaces in question have the same polar space, and the first line then
clearly follows if one switches $\,V\hs$ with $\,V\hn^*\nnh$.

If a $\,2$-form $\,\sy\hs$ in $\,V\hs$ is 
nondegenerate ($\mathrm{Ker}\,\sy\nh=\{0\}$), and so $\,n=2m\,$ is even,
\begin{equation}\label{dfd}
\begin{array}{l}
\mathrm{we\ define\ \text{\it the\ dual\/}\ of\ }\,\sy\hs\mathrm{\ to\ be\
the\ }\,(n-2)\hyp\mathrm{form\ }\,\my=\omega\beta\,\mathrm{\ for\ the\
vol}\hyp\\
\mathrm{ume\ form\ }\,\omega=(m\hh!)^{-\nh1}\sy\hh^{\wedge m}\nh\mathrm{,\
where\ }\,\beta\,\mathrm{\ is\ the\ bi\-vec\-tor\ reciprocal\ to\ }\,\sy.
\end{array}
\end{equation}
Thus, under the assumption (\ref{bas}),
\begin{equation}\label{mds}
\begin{array}{l}
\mathrm{the\ dual\ of\ }\,\sy\,=\,\sy\nh_1\w+\ldots+\,\sy\nnh_m\w\mathrm{\,\
with\ }\hs\,\sy\nnh_i\w\nh=\hs\xi^{2i-1}\nh\wedge\,\xi^{2i}\,\mathrm{\
equals}\\
\my=\my_1\w+\ldots+\my_m\w\hs\mathrm{\ for\ }\,\my_i\w
=-\sy\nh_1\w\wedge\ldots\wedge\sy\nnh_{i-1}\w\nnh\wedge\sy\nnh_{i+1}\w\nnh
\wedge\ldots\wedge\sy\nnh_m\w.
\end{array}
\end{equation}
In fact, the bi\-vec\-tor 
reciprocal to $\,\sy\,$ is $\,\beta
=-e_1\w\wedge e_2\w-\ldots-e_{n-1}\w\wedge e_n\w$, so that
$\,\omega\beta=\my\,$ for 
$\,\omega=(m\hh!)^{-\nh1}\sy\hh^{\wedge m}\nh
=\xi^1\nnh\nh\wedge\ldots\wedge\hs\xi^n\nh$, as 
$\,\omega[\hs e_{2i-1}\w\wedge e\hn_i\w]=-\my_i\w$. By (\ref{mds}),
\begin{equation}\label{led}
\mathrm{the\ dual\ of\ }\,\sy\hs\mathrm{\ equals\ }-\nnh\nnh1\,\mathrm{\ if\ 
}\,n=2\,\mathrm{\ and \ }-\nnh\sy\hs\mathrm{\ when\ }\,n=4\hh.
\end{equation}
\begin{remark}\label{evenq}Let $\,\my\,$ be a nonzero exterior $\,(n-2)$-form 
in an $\,n$-di\-men\-sion\-al real vector space $\,V\nnh$. For the kernel 
$\,Z\,$ and divisibility space $\,D\,$ of $\,\my$,
(\ref{prn}) and \hbox{(\ref{rnk}-a)} give $\,k=\dim Z\in\{0,2\}$, whereas 
$\,q=\dim D\,$ is even and $\,2\le q\le n\,$ as a consequence of (\ref{img})
and Lemma~\ref{divis}(a) for $\,\px=n-2\,$ and $\,s=n-q$.
\end{remark}
\begin{lemma}\label{duali}Any\/ $\,(n-2)$-form in\/ $\,V\nnh$, with\/ 
$\,n\hs=\hs2m$, dual to a nondegenerate \hbox{$\,2$-form}, is
in\-di\-vis\-i\-ble, as in\/ {\rm(\ref{ind})}. 
Conversely, any in\-di\-vis\-i\-ble\/ $\,(n-2)$-form\/ $\,\my\,$ is equal or
opposite to the dual of a nondegenerate\/ $\,2$-form\/ $\,\sy\nh$, and\/
$\,\my\,$ determines such\/ $\,\sy\hn$ uniquely up to a sign. For even\/ $\,m$,
the phrases `or opposite to\nnh' and \hbox{`up to a sign\nh'} may be deleted.
Furthermore, $\,\pm\hh\sy\hs$ has an explicit expression in terms of\/ $\,\my$.
\end{lemma}
\begin{proof}With the index ranges $\,i=1,\dots,m\,$ and $\,k=1,\dots,n=2m$, 
let $\,\my\,$ be dual to $\,\sy$. For 
$\,\xi=a\hn_k\w\hs\xi\hh^k\nh\in V\hn^*$ and 
$\,\theta\hn_k\w=
\xi^1\nnh\wedge\ldots\wedge\hs\xi^{k-1}\nnh\wedge\hs\xi^{k+1}\nnh\wedge
\ldots\wedge\hs\xi^n\nh$, where we use (\ref{mds}), $\,\xi\wedge\my\,$ 
equals the combination of $\,\theta\hn_{2i}\w$ and
$\,\theta\hn_{2i-1}\w$ with the coefficients $\,a_{2i-1}\w$ and $\,a_{2i}\w$.
Linear independence of all $\,\theta\hn_k\w$ now gives $\,\xi=0\,$
whenever $\,\xi\wedge\my=0$, proving our first claim. Also,
$\,\my\,$ uniquely determines the reciprocal $\,\beta\,$ of $\,\sy$,
and hence $\,\sy\hs$ itself, up to a nonzero scalar factor: 
if $\,\my=\omega\beta=*\beta\,$ for {\it any\/} volume form $\,\omega$,
(\ref{inv}) with $\,\px=2\,$ gives $\,\beta=*\my$.

Conversely, 
let $\,\my\in[V\hn^*]^{\wedge(n-2)}$ be in\-di\-vis\-i\-ble. 
With $\,\omega\,$ and $\,\alpha\,$ as in the lines preceding (\ref{dua}), 
$\,\alpha\my\,$ is -- by (\ref{img}) -- a nondegenerate 
bi\-vec\-tor, and so it has a reciprocal nondegenerate $\,2$-form
$\,\lambda$. Hence 
$\,(m\hh!)^{-\nh1}\lambda\nh^{\wedge m}$ equals $\,\kappa\hskip1pt\omega$
for some $\,\kappa\in\bbR$. Replacing $\,\omega\,$ with
$\,\widetilde\omega=c\hs\omega\,$ 
leads to $\,\widetilde\alpha\,$ and $\,\widetilde\lambda\,$ equal, 
respectively, to $\,c^{\nh-1}\alpha\,$ and $\,c\lambda$, so that
$\,(m\hh!)^{-\nh1}\widetilde\lambda\nh^{\wedge m}\nh
=\widetilde\kappa\hskip1.2pt\widetilde\omega$, where
$\,\widetilde\kappa=c\hh^{m-1}\nh\kappa$. Some
choice of $\,c$, unique up to a sign, now gives $\,|\widetilde\kappa|=1\,$
and, if $\,m\,$ is even, a unique $\,c\,$ yields $\,\widetilde\kappa=1$.
The $\,2$-form $\,\sy\hn=\widetilde\lambda=c\lambda$, for this $\,c$, has
the reciprocal bi\-vec\-tor $\,\beta=c^{\nh-1}\alpha\my\,$ and
$\,(m\hh!)^{-\nh1}\sy\hh^{\wedge m}\nh=\pm\hs\widetilde\omega
=\pm\hs c\hs\omega$. Thus, by (\ref{inv}), 
$\,\pm\hs\widetilde\omega\beta=\pm\hs\omega\hh\alpha\my
=\pm\my\,$ with the required sign $\,\pm\hh$.
\end{proof}
By Lemma~\ref{duali}, in\-di\-vis\-i\-ble $\,(n-2)$-forms exist only in even
dimensions $\,n$.

We use the term {\it duality\/} for the natural bijective 
correspondence, established in Lemma~\ref{duali}, for even dimensions $\,n$,
between nondegenerate 
exterior $\,2$-forms $\,\sy$ and in\-di\-vis\-i\-ble $\,(n-2)$-forms
$\,\my$, with both $\,\sy\nh,\my\,$ only defined up to a sign.

\section{Exterior $\,3$-forms in dimension six}\label{ds}
\setcounter{equation}{0}
The algebraic classification of exterior $\,3$-forms $\,\my\,$ in a
six-di\-men\-sion\-al real vector space $\,V\nnh$, 
known since Reichel's 1907 thesis \cite{reichel}, is copied here, with minor 
changes, from 
Bryant's paper \cite[p.\,599]{bryant}: 
the possible (nonzero) types 
appear as 
\begin{equation}\label{six}
\begin{array}{rl}
\mathrm{a)}&\my\hs\,=\,\hs\xi^1\nnh\nh\wedge\hs\xi^2\hn\nnh\wedge\hs\xi^3\hs
+\,\hs\xi^3\nnh\nh\wedge\hs\xi\hs^4\nh\nnh\wedge\hs\xi\hh^5\hs
+\,\hs\xi\hh^5\nnh\nh\wedge\hs\xi\hh^6\hn\nnh\wedge\hs\xi^1\hs
+\,\hs\xi^2\nnh\nh\wedge\hs\xi\hs^4\nh\nnh\wedge\hs\xi\hh^6\hn,\\
\mathrm{b)}&\my\hs\,=\,\hs\xi^1\nnh\nh\wedge\hs\xi^2\hn\nnh\wedge\hs\xi^3\hs
+\,\hs\xi^3\nnh\nh\wedge\hs\xi\hs^4\nh\nnh\wedge\hs\xi\hh^5\hs
+\,\hs\xi\hh^5\nnh\nh\wedge\hs\xi\hh^6\hn\nnh\wedge\hs\xi^1\hn,\\
\mathrm{c)}&\my\hs\,=\,\hs\xi^1\nnh\nh\wedge\hs\xi^2\hn\nnh\wedge\hs\xi^3\hs
+\,\hs\xi\hs^4\nnh\nh\wedge\hs\xi\hh^5\nh\nnh\wedge\hs\xi\hh^6\hn,\\
\mathrm{d)}&\my\hs\,=\,\hs\xi^1\nnh\nh\wedge\hs\xi^2\hn\nnh\wedge\hs\xi^3\hs
+\,\hs\xi\hs^4\nnh\nh\wedge\hs\xi\hh^5\nh\nnh\wedge\hs\xi^1\hn,\\
\mathrm{e)}&\my\hs\,=\,\hs\xi^1\nnh\nh\wedge\hs\xi^2\hn\nnh\wedge\hs\xi^3
\end{array}
\end{equation}
in some basis $\,\xi^1\nnh,\dots,\hs\xi\hh^6$ of $\,V\hn^*\nnh$, dual to a
basis $\,e_1\w,\dots,e_6\w$ of $\,V\nnh$. 
Following Hitchin \cite[pp.\,551-552]{hitchin}, we call 
(\ref{six}-a) and (\ref{six}-c) the {\it com\-plex$\hs/\nh$real stable
cases}.

The five types (\ref{six}) are illustrated by the diagrams on p.\,11.

Each of the first four types (\ref{six}) has an associated 
pair of invariants:
\begin{equation}\label{fiv}
\mathrm{a)}\hskip3pt\pm J\,\mathrm{\ and\
}\,\my(J\hs\cdot\,,\,\cdot\,,\,\cdot\,)\hh,\hskip8pt
\mathrm{b)}\hskip3ptH\,\mathrm{\ and\ }\,\varTheta\hh,\hskip8pt
\mathrm{c)}\hskip3ptH\hn^\pm\mathrm{\ and\ }\,\eta^\pm\nh,\hskip8pt
\mathrm{d)}\hskip3ptD\,\mathrm{\ and\ }\,\bbR\hh\zeta\hh,
\end{equation}
reflecting their stabilizer groups; see the exposition by Bryant
\cite[p.\,602, 
Remark 31]{bryant}. For the reader's convenience, our presentation of 
(\ref{fiv}) is self-con\-tain\-ed. The invariant character of the objects in
question is, in each case, due to their being uniquely determined by
$\,\my$.

In the complex stable case (\ref{six}-a), as pointed out 
by Hitchin \cite[p.\,552]{hitchin}, 
\begin{equation}\label{rlp}
\my\,\,\mathrm{\ equals\ the\ real\ part\ of\ }\,\,\omega\,=\,
(\xi\hs^4\nh+\,i\hs\xi^1)\wedge(\xi\hh^6\nh+\,i\hs\xi^3)\wedge(\xi^2\nh+\,i\hs\xi\hh^5).
\end{equation}
Our basis $\,\xi^1\nnh,\dots,\hs\xi\hh^6$ of $\,V\hn^*$ is dual to a basis 
$\,e_1\w,\dots,e_6\w$ of $\,V\nnh$. Setting
\begin{equation}\label{acs}
J\hn e\hn_4\w=\,e_1\w\hh,\qquad J\hn e_6\w=\,e_3\w\hh,\qquad J\hn e_2\w
=\,e_5\w\hh,
\end{equation}
we define a com\-plex-struc\-ture tensor $\,J:V\nh\to V\nnh$, making the
factor complex $\,1$-forms 
in (\ref{rlp}) com\-plex-lin\-e\-ar, so that $\,\wedge\,$ in
(\ref{rlp}) is also the complex exterior product,
\begin{equation}\label{tri}
\begin{array}{l}
\omega\,\mathrm{\ in\ (\ref{rlp})\ is\ com\-plex}\hyp\mathrm{tri\-lin\-e\-ar,\ 
and\ so\ }\,\my(J\hs\cdot\,,J\hs\cdot\,,\,\cdot\,)=-\my\hh,\\
\mathrm{which\ in\ turn\ implies\ total\ skew}\hyp\mathrm{sym\-me\-try\ of\ 
}\,\hs\my(J\hs\cdot\,,\,\cdot\,,\,\cdot\,)\hh.
\end{array}
\end{equation}
As shown by Bryant \cite[p.\,596]{bryant},
\begin{equation}\label{unq}
\pm J\,\mathrm{\ are\ the\ only\ com\-plex}\hyp\mathrm{struc\-ture\ tensors\ 
with\ }\,\my(J\hs\cdot\,,J\hs\cdot\,,\,\cdot\,)=-\my\hh.
\end{equation}
We now justify (\ref{unq}). Let 
$\,\my_k\w=\my(e\hn_k\w,\,\cdot\,,\,\cdot\,)\,$ and
$\,\xi^{jk}\nh=\xi^j\nnh\nh\wedge\hs\xi^k\nh$. From (\ref{six}-a), 
\begin{equation}\label{moe}
\my_1\w=\hs\xi^{23}\nh+\hs\,\xi\hh^{56}\nh,\quad
\my_2\w=\hs\xi^{31}\nh+\hs\,\xi^{46}\nh,\quad
\my_4\w=\hs\xi\hh^{53}\nh+\hs\,\xi^{62}\nh,\quad
\my_5\w=\hs\xi^{34}\nh+\hs\,\xi^{61}\nh.
\end{equation}
{\it Assuming just that\/} $\,\my(J\hs\cdot\,,J\hs\cdot\,,\,\cdot\,)=-\my$,
we get $\,\my(\,\cdot\,,J\hs\cdot\,,\,\cdot\,)
=\my(J\hs\cdot\,,\,\cdot\,,\,\cdot\,)$. Thus,
\begin{equation}\label{myj}
\my(e_2\w,J\hn e_1\w,\,\cdot\,)\,=\,\my(J\hn e_2\w,e_1\w,\,\cdot\,)\hh,\qquad
\my(e\hn_4\w,J\hn e_2\w,\,\cdot\,)\,=\,\my(J\hn e\hn_4\w,e_2\w,\,\cdot\,)\hh,
\end{equation}
and -- for the same reason -- as $\,\my\,$ is skew-sym\-met\-ric,
$\,\my(Jv,v,\,\cdot\,)=0\,$ for any vector field $\,v$. Applied to the
vector fields $\,e_1\w,\dots,e_6\w$, this implies that each
$\,J\hn e_k\w$ is a section of $\,\mathrm{Ker}\,\my_k\w$, and so, by
(\ref{moe}), (\ref{dim}) and Remark~\ref{imges}, 
the spans of $\,\{e_1\w,e\hn_4\w\}\,$ and 
$\,\{e_2\w,e_5\w\}\,$ are $\,J\hn$-in\-var\-i\-ant. Writing
$\,J\hn e_1\w=a\hs e_1\w+c\hs e\hn_4\w$,
$\,J\hn e\hn_4\w=\widetilde c\hs e_1\w+\widetilde a\hs e\hn_4\w$,
and $\,J\hn e_2\w=a_{22}\w e_2\w+a_{25}\w e_5\w$, then using (\ref{moe})
and (\ref{myj}), we obtain
\[
-a\hs\xi^3\nh+c\hs\xi\hh^6\nh=-a_{22}\w\hs\xi^3\nh-a_{25}\w\hs\xi\hh^6\nh,\qquad
-a_{22}\w\hs\xi\hh^6\nh+a_{25}\w\hs\xi^3\nh
=\widetilde c\hs\xi^3\nh-\widetilde a\hs\xi\hh^6\nh.
\]
Thus,
$\,(\widetilde a,\widetilde c)=(a,-c)\,$ and the matrix of $\,J\,$
in $\,\mathrm{Span}\hs(e_1\w,e\hn_4\w)\,$ has the rows $\,(a,-c)$,
$\,(c,a)$, making it conjugate to the multiplication by $\,a+c\hh i\,$ in
$\,\bbC$, so that $\,(a,c)=(0,\mp1)\,$ and 
$\,J\hn e\hn_4\w\nh=\pm\hs e_1\w$. As (\ref{six}-a) is invariant under
simultaneous cyclic permutations of $\,(1,3,5)\,$ and $\,(2,4,6)$,
$\,J\hn e_6\w\nh=\pm\hs e_3\w$ and $\,J\hn e_2\w\nh=\pm\hs e_5\w$. The
three $\,\pm\,$ signs must all be the same, since the last displayed formula,
with $\,(a,c,\widetilde a,\widetilde c)=(0,\mp1,0,\pm1)$ and (therefore)
$\,a_{22}\w=0$, gives $\,a_{25}\w=-c=\pm1$. The third sign falls in line 
due to the aforementioned cy\-clic-per\-mu\-ta\-tion invariance, implying 
(\ref{unq}).

Next, (\ref{six}-b) implies that, by Lemma~\ref{prddc}(i), 
the sub\-space
$\,H'\nh=\mathrm{Span}\hs(\xi^1\nnh,\hs\xi^3\nnh,\hs\xi\hh^5)$ of 
$\,V\hn^*$ equals
$\,\{\xi\in V\hn^*:\xi\wedge\my\,\mathrm{\ is\ de\-com\-pos\-able}\}$.
With $\,e_1\w,\dots,e_6\w$ dual to $\,\xi^1\nnh,\dots,\hs\xi\hh^6$ as before, 
we define $\,H=\mathrm{Span}\hs(e_2\w,e\hn_4\w,e_6\w)\subseteq V\hs$ to be 
the polar space of $\,H'\nh$. 
Setting $\,\varTheta v=\my(v,\,\cdot\,,\,\cdot\,)\,$ we now obtain 
a linear iso\-mor\-phism $\,\varTheta:H\to[H'\hh]^{\wedge2}\nnh$. 
In fact, $\,e_2\w,e\hn_4\w,e_6\w$ form 
a basis of $\,H$, while, by (\ref{six}-b),
\begin{equation}\label{the}
\varTheta e_2\w=\hs\xi^3\nnh\nnh\wedge\hs\xi^1\nh,\qquad
\varTheta e\hn_4\w=\hs\xi\hh^5\nnh\nnh\wedge\hs\xi^3\nh\qquad
\varTheta e_6\w=\hs\xi^1\nnh\nnh\wedge\hs\xi\hh^5\nh.
\end{equation}
Suppose next that we have (\ref{six}-c). By Lemma~\ref{prddc}(ii), 
the set of all $\,1$-forms $\,\xi\,$ such that $\,\xi\wedge\my\,$ is 
de\-com\-pos\-able is the union of the sub\-spaces
$\,\mathrm{Span}\hs(\xi^1\nnh,\hs\xi^2\nnh,\hs\xi^3)\,$ and
$\,\mathrm{Span}\hs(\xi\hs^4\nnh,\hs\xi\hh^5\nnh,\hs\xi\hh^6)\,$ of 
$\,V\hn^*$. The resulting unordered pair $\,\{H\hn^+\nnh\nh,\,H\hn^-\}\,$ of
their polar spaces in $\,V\hs$ is therefore uniquely determined by 
$\,\my\,$ and, consequently, so is the unordered pair
$\,\{\eta^+\nnh,\eta^-\}\,$ of the $\,3$-forms 
$\,\eta^+\nh=\hs\xi^1\nnh\nh\wedge\hs\xi^2\hn\nnh\wedge\hs\xi^3$ and 
$\,\eta^-\nh=\hs\xi\hs^4\nnh\nh\wedge\hs\xi\hh^5\nh\nnh\wedge\hs\xi\hh^6\nh$.
Note that $\,H^+\nh=\mathrm{Span}\hs(e\hn_4\w,e_5\w,e_6\w)\,$ and
$\,H^-\nh=\mathrm{Span}\hs(e_1\w,e_2\w,e_3\w)$.

In the case (\ref{six}-d), $\,\my\,$ is not de\-com\-pos\-able
(Remark~\ref{nondc}). 
The vector sub\-space $\,D\hh'$ of $\,V\hn^*\nh$, polar to the divisibility
space $\,D\,$ of $\,\my$, is thus the line spanned by $\,\xi^1\nh$, as 
Lemma~\ref{divis}(a), with $\,\px=3\,$ and $\,s\ge1$, gives $\,s=1$. By
(\ref{six}-d), Lemma~\ref{divis}(c) now applies to $\,s=1$ and
$\,\zeta=\hs\xi^2\hn\nnh\wedge\hs\xi^3\nnh+\,\xi\hs^4\nh\nnh\wedge\hs\xi\hh^5\nh$, 
turning $\,\zeta\in[D^*]^{\wedge2}$ into an in\-di\-vis\-i\-ble factor
of $\,\my$, 
defined (up to multiplications by scalars) as in the lines preceding
(\ref{vol}).
\begin{figure}[H]
\centering
\input{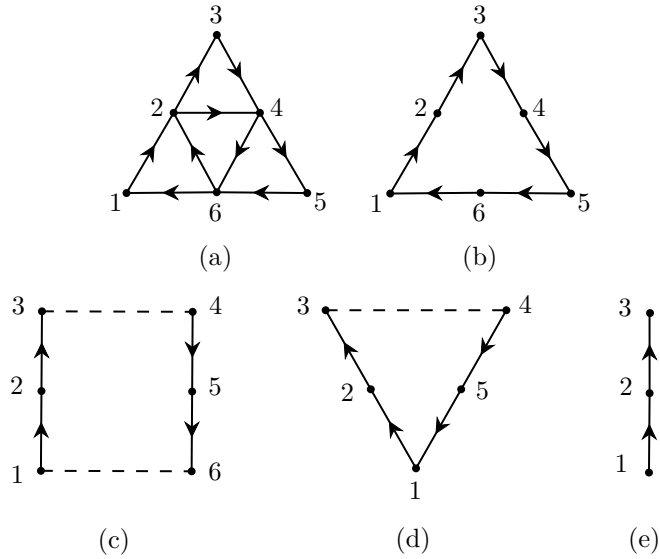}
\vspace{1em}
\caption{The five types (\ref{six}) of $3$-forms in dimension six. Each maximal
solid line segment corresponds to one summand in (\ref{six}), and so does
the small inscribed \vbox{\hbox{$\bigtriangledown$}\vskip-.2pt} triangle in
(a). They are all oriented as indicated by the arrows.}
\label{fig:3-forms-dim-6}
\end{figure}
\vskip-34pt
\begin{figure}[H]
  \centering
  \tikzset{every picture/.style={line width=0.75pt}} 

\begin{tikzpicture}[x=0.75pt,y=0.75pt,yscale=-1,xscale=1]

\draw  [dash pattern={on 4.5pt off 4.5pt}]  (136.5,230.25) -- (123.8,159.4) ;
\draw [shift={(129.27,189.9)}, rotate = 79.84] [fill={rgb, 255:red, 0; green, 0; blue, 0 }  ][line width=0.08]  [draw opacity=0] (10.72,-5.15) -- (0,0) -- (10.72,5.15) -- (7.12,0) -- cycle    ;
\draw  [dash pattern={on 4.5pt off 4.5pt}]  (137.25,234.25) -- (143.4,268.9) ;
\draw  [fill={rgb, 255:red, 0; green, 0; blue, 0 }  ,fill opacity=1 ] (122.3,159.4) .. controls (122.3,158.57) and (122.97,157.9) .. (123.8,157.9) .. controls (124.63,157.9) and (125.3,158.57) .. (125.3,159.4) .. controls (125.3,160.23) and (124.63,160.9) .. (123.8,160.9) .. controls (122.97,160.9) and (122.3,160.23) .. (122.3,159.4) -- cycle ;
\draw  [fill={rgb, 255:red, 0; green, 0; blue, 0 }  ,fill opacity=1 ] (141.5,103.4) .. controls (141.5,102.57) and (142.17,101.9) .. (143,101.9) .. controls (143.83,101.9) and (144.5,102.57) .. (144.5,103.4) .. controls (144.5,104.23) and (143.83,104.9) .. (143,104.9) .. controls (142.17,104.9) and (141.5,104.23) .. (141.5,103.4) -- cycle ;
\draw    (143,103.4) -- (60.88,211.4) ;
\draw [shift={(105.87,152.23)}, rotate = 127.25] [fill={rgb, 255:red, 0; green, 0; blue, 0 }  ][line width=0.08]  [draw opacity=0] (10.72,-5.15) -- (0,0) -- (10.72,5.15) -- (7.12,0) -- cycle    ;
\draw    (147.67,129.67) -- (161.1,211.4) ;
\draw [shift={(155.19,175.47)}, rotate = 260.67] [fill={rgb, 255:red, 0; green, 0; blue, 0 }  ][line width=0.08]  [draw opacity=0] (10.72,-5.15) -- (0,0) -- (10.72,5.15) -- (7.12,0) -- cycle    ;
\draw    (143,103.4) -- (225.52,159.4) ;
\draw [shift={(188.4,134.21)}, rotate = 214.16] [fill={rgb, 255:red, 0; green, 0; blue, 0 }  ][line width=0.08]  [draw opacity=0] (10.72,-5.15) -- (0,0) -- (10.72,5.15) -- (7.12,0) -- cycle    ;
\draw  [dash pattern={on 4.5pt off 4.5pt}]  (143,103.4) -- (123.8,159.4) ;
\draw [shift={(135.51,125.25)}, rotate = 108.92] [fill={rgb, 255:red, 0; green, 0; blue, 0 }  ][line width=0.08]  [draw opacity=0] (10.72,-5.15) -- (0,0) -- (10.72,5.15) -- (7.12,0) -- cycle    ;
\draw  [fill={rgb, 255:red, 0; green, 0; blue, 0 }  ,fill opacity=1 ] (141.9,268.9) .. controls (141.9,268.07) and (142.57,267.4) .. (143.4,267.4) .. controls (144.23,267.4) and (144.9,268.07) .. (144.9,268.9) .. controls (144.9,269.73) and (144.23,270.4) .. (143.4,270.4) .. controls (142.57,270.4) and (141.9,269.73) .. (141.9,268.9) -- cycle ;
\draw    (143.4,268.9) -- (161.1,211.4) ;
\draw [shift={(150.34,246.36)}, rotate = 287.11] [fill={rgb, 255:red, 0; green, 0; blue, 0 }  ][line width=0.08]  [draw opacity=0] (10.72,-5.15) -- (0,0) -- (10.72,5.15) -- (7.12,0) -- cycle    ;
\draw    (143.4,268.9) -- (225.52,159.4) ;
\draw [shift={(180.56,219.35)}, rotate = 306.87] [fill={rgb, 255:red, 0; green, 0; blue, 0 }  ][line width=0.08]  [draw opacity=0] (10.72,-5.15) -- (0,0) -- (10.72,5.15) -- (7.12,0) -- cycle    ;
\draw    (143.4,268.9) -- (60.88,211.4) ;
\draw [shift={(107.47,243.87)}, rotate = 214.87] [fill={rgb, 255:red, 0; green, 0; blue, 0 }  ][line width=0.08]  [draw opacity=0] (10.72,-5.15) -- (0,0) -- (10.72,5.15) -- (7.12,0) -- cycle    ;
\draw    (60.88,211.4) -- (162.6,211.4) ;
\draw [shift={(105.24,211.4)}, rotate = 0] [fill={rgb, 255:red, 0; green, 0; blue, 0 }  ][line width=0.08]  [draw opacity=0] (10.72,-5.15) -- (0,0) -- (10.72,5.15) -- (7.12,0) -- cycle    ;
\draw    (161.1,211.4) -- (225.52,159.4) ;
\draw [shift={(197.2,182.26)}, rotate = 141.09] [fill={rgb, 255:red, 0; green, 0; blue, 0 }  ][line width=0.08]  [draw opacity=0] (10.72,-5.15) -- (0,0) -- (10.72,5.15) -- (7.12,0) -- cycle    ;
\draw  [dash pattern={on 4.5pt off 4.5pt}]  (60.88,211.4) -- (123.8,159.4) ;
\draw [shift={(87.33,189.54)}, rotate = 320.43] [fill={rgb, 255:red, 0; green, 0; blue, 0 }  ][line width=0.08]  [draw opacity=0] (10.72,-5.15) -- (0,0) -- (10.72,5.15) -- (7.12,0) -- cycle    ;
\draw  [dash pattern={on 4.5pt off 4.5pt}]  (123.8,159.4) -- (225.52,159.4) ;
\draw [shift={(168.16,159.4)}, rotate = 0] [fill={rgb, 255:red, 0; green, 0; blue, 0 }  ][line width=0.08]  [draw opacity=0] (10.72,-5.15) -- (0,0) -- (10.72,5.15) -- (7.12,0) -- cycle    ;
\draw  [fill={rgb, 255:red, 0; green, 0; blue, 0 }  ,fill opacity=1 ] (224.02,159.4) .. controls (224.02,158.57) and (224.7,157.9) .. (225.52,157.9) .. controls (226.35,157.9) and (227.02,158.57) .. (227.02,159.4) .. controls (227.02,160.23) and (226.35,160.9) .. (225.52,160.9) .. controls (224.7,160.9) and (224.02,160.23) .. (224.02,159.4) -- cycle ;
\draw  [fill={rgb, 255:red, 0; green, 0; blue, 0 }  ,fill opacity=1 ] (159.6,211.4) .. controls (159.6,210.57) and (160.27,209.9) .. (161.1,209.9) .. controls (161.93,209.9) and (162.6,210.57) .. (162.6,211.4) .. controls (162.6,212.23) and (161.93,212.9) .. (161.1,212.9) .. controls (160.27,212.9) and (159.6,212.23) .. (159.6,211.4) -- cycle ;
\draw  [fill={rgb, 255:red, 0; green, 0; blue, 0 }  ,fill opacity=1 ] (59.38,211.4) .. controls (59.38,210.57) and (60.05,209.9) .. (60.88,209.9) .. controls (61.7,209.9) and (62.38,210.57) .. (62.38,211.4) .. controls (62.38,212.23) and (61.7,212.9) .. (60.88,212.9) .. controls (60.05,212.9) and (59.38,212.23) .. (59.38,211.4) -- cycle ;
\draw    (143,103.4) -- (147.67,129.67) ;

\draw (45.6,205.6) node [anchor=north west][inner sep=0.75pt]    {$1$};
\draw (137.6,85) node [anchor=north west][inner sep=0.75pt]    {$2$};
\draw (163.6,211) node [anchor=north west][inner sep=0.75pt]    {$3$};
\draw (230.4,151.6) node [anchor=north west][inner sep=0.75pt]    {$4$};
\draw (138.8,272.6) node [anchor=north west][inner sep=0.75pt]    {$5$};
\draw (112.4,144.4) node [anchor=north west][inner sep=0.75pt]    {$6$};

\end{tikzpicture}
  \vspace{1em}
  \caption{The octahedron version of (\ref{six}-a). Four mutually
non-ad\-ja\-cent faces correspond to $\,\mu=\mathrm{Re}\,\hh\omega\,$ and
the remaining four 
to $\,\mathrm{Im}\,\hs\omega$, where seven faces (all but one of the latter, 
namely, 531) represent the same orientation of the boundary surface. The four
$\,\mu$-faces are also
characterized by being coherently oriented by the ar\-row-mark\-ed
orientations of their sides.}
\label{fig:octahedron}
\end{figure}
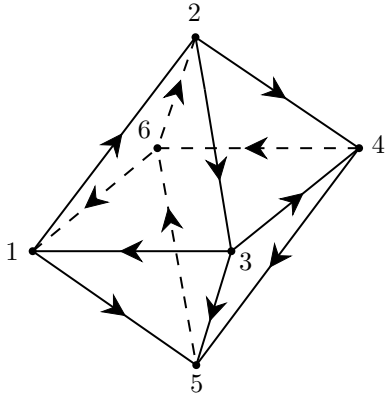

\section{The simplest invariants of differential forms}\label{si}
\setcounter{equation}{0}
Given a manifold $\,M\nh$, we say that vector sub\-bundles
$\,\dz\,$ of $\,T\nh M$ and $\,\ez\hs$ of $\,T^*\hskip-2.1ptM$ are
{\it polar to each other\/} when each is the other's ``orthogonal
complement'' relative to the obvious pairing between tangent vectors and
$\,1$-forms.

For an algebraically constant differential $\,\px\hs$-form $\,\my\,$ on an 
$\,n$-di\-men\-sion\-al manifold $\,M\hh$ it obviously follows that,
with $\,\,\dz\nnh_x\w$ denoting the divisibility space of $\,\my_x\w$,
\begin{equation}\label{cst}
\begin{array}{rl}
\mathrm{a)}&\mathrm{the\ function\ 
}\,x\mapsto\mathrm{rank}\hskip2.7pt\my_x\w\mathrm{\ is\ constant\ on\
}M\nh,\\
\mathrm{b)}&\dz\nnh_x\w\hs\mathrm{\ has\ the\ same\ dimension\
}\hs\,q\,\hh\mathrm{\ at\ all\ }\,x\in M\nh.
\end{array}
\end{equation}
The case (\ref{cst}-a), or  (\ref{cst}-b), gives rise to the natural
distribution $\,\zz$, or $\,\dz$, on $\,M\nh$, obtained by declaring
$\,\zz\nh_x\w$, or $\,\dz\nnh_x\w$, at any $\,x\in M\nh$, to be the
sub\-space $\,Z\,$ or, respectively, $\,D\,$
of $\,V\nh=\txm\,$ associated with $\,\my=\my_x\w$ as in Sect.~\ref{ie}.
We call $\,\zz\,$ and $\,\dz\,$ the {\it kernel\/} and the {\it divisibility
distribution\/} of $\,\my$. Thus, 
$\,\dz\,$ is 
polar to the sub\-bundle $\,\dz'$ of $\,T^*\hskip-2.1ptM$ having as sections 
the $\,1$-forms $\,\xi\,$ with $\,\xi\wedge\my=0$. Smoothness of $\,\zz\,$ and
$\,\dz$, under the respective assumptions (\ref{cst}-a) or (\ref{cst}-b),
follows
since $\,\zz$, or $\,\dz'\nh$, is the the kernel of a con\-stant-rank
vec\-tor-bun\-dle mor\-phism: the former, from $\,T\nh M\,$ to
$\,[T^*\hskip-2.1ptM]^{\wedge(\px-1)}\nnh$, sending $\,v\,$ to
$\,\my(v,\,\cdot\,,\dots,\,\cdot\,)$, the latter defined analogously, just
with 
$\,T\nh M\,$ and $\,T^*\hskip-2.1ptM\,$ switched; see (\ref{img}) and
(\ref{imk}). By (\ref{rnk}-b),
\begin{equation}\label{zid}
\mathrm{if\ (\ref{cst}}\hyp\mathrm{a)\ and\ (\ref{cst}}\hyp\mathrm{b)\ both\
hold,\ 
}\,\,\zz\,\subseteq\,\dz\,\mathrm{\ \ unless\ \ }\,\my=0.
\end{equation}
When $\,d\hh\my=0$, the distribution $\,\zz\,$ is easily 
seen to be in\-te\-gra\-ble (Lemma~\ref{prjct} below). However, $\hs\dz\hs$
need not be: see \cite[Sect.\,12]{derdzinski-piccione-terek}.

By a {\it volume form\/} on a manifold we mean a
no\-where\hh-zero top degree differential form, which amounts to a function 
without zeros in dimension $\,0$.

For a nonzero 
differential $\,\px\hs$-form $\,\my\,$ satisfying (\ref{cst}-b) 
on an $\,n$-di\-men\-sion\-al manifold $\,M\,$ and 
its divisibility distribution $\,\dz$, let $\,s=n-q\,$ in 
Lemma~\ref{divis}(a), so that $\,s\le\px\,$ and, by Remark~\ref{evenq},
\begin{equation}\label{qis}
\dz\hs\,\mathrm{\ has\ some\ even\ fibre\ dimension\ 
}\, q\in\{2,\dots,n\}\,\mathrm{\ \ if\ \ }\,\px\,=\,n\,-\,2\hh.
\end{equation}
Whether or not $\,\px=n-2$, assuming in\-te\-gra\-bi\-li\-ty of $\,\dz$, we
now replace $\,M\,$ with a sufficiently small neighborhood of any given point
so as to make $\,\dz\,$ the vertical distribution of a fibration 
$\,\pi:M\to\varSigma$, which gives rise to
\begin{equation}\label{vls}
\begin{array}{l}
\mathrm{the\ }\,\pi\hyp\mathrm{pull\-back\ 
}\,\xi^1\nnh\nh\wedge\ldots\wedge\hs\xi^s\mathrm{\ of\ a\ volume\ 
}\hs s\hyp\mathrm{form\ on\ }\,\varSigma\mathrm{,\ for\ the}\\
\mathrm{pull}\hyp\mathrm{backs\ }\hs\xi^1\nnh,\dots,\xi^s\nh\mathrm{\
under\ }\hs\pi\hs\mathrm{\ of\ some\ }\hn1\hyp\mathrm{forms\ trivializing\ 
}\hs T\hskip.2pt^*\hskip-2pt\varSigma,\\
\mathrm{and\ then\ }\,\my\,
=\,\hs\xi^1\nnh\nh\wedge\ldots\wedge\hs\xi^s\nnh\nh\wedge\hs\zeta\hs\mathrm{\
for\ some\ }\,(\px-\nh s)\hyp\mathrm{form\ }\,\hs\zeta\hs\mathrm{\ on\ }\,M\nh,
\end{array}
\end{equation}
the last line due to Lemma~\ref{divis}(b) applied to
$\,\my\,$ and $\,\xi^1\nnh,\dots,\hs\xi^s$ at any point $\,x\in M\nh$, with 
$\,D=\dz\nnh_x\w$ and $\,V\nh=\txm\nh$. By 
Lemma~\ref{divis}(c), the restriction of $\,\zeta$ to each leaf $\,L\,$ of
$\,\dz\,$ is uniquely determined by our $\,\my\,$ and 
$\,\xi^1\nnh,\dots,\hs\xi^s\nnh$. Replacing $\,\xi^1\nnh,\dots,\hs\xi^s$
by another such $\,s$-tuple of $\,\pi$-pull\-backs causes
$\,\xi^1\nnh\nnh\wedge\ldots\wedge\hs\xi^s$ to be replaced with its product by 
a function {\it constant along the leaves of\/} $\,\dz$. Thus, the 
restriction of $\,\zeta\,$ to each leaf $\,L\,$ is {\it unique up to 
multiplications by nonzero constants}, and 
\begin{equation}\label{idf}
\mathrm{we\ call\ this\ }\,(\px-\nh s)\hyp\mathrm{form\ }\,\zeta\hs\mathrm{\
the\ }\text{\it in\-di\-vis\-i\-ble factor\/}\mathrm{\ of\
}\,\my\,\mathrm{\ on\ the\ leaf\ }\hs L\hh.
\end{equation}
\begin{lemma}\label{cllvs}For a nonzero closed differential\/ $\,\px$-form\/
$\,\my\,$ on an\/ $\,n$-di\-men\-sion\-al manifold\/ $\,M\nnh$,
satisfying the condition\/ {\rm(\ref{cst}-b)} and having the divisibility
distribution\/ $\,\dz\,$ of 
co\-dimen\-sion\/ $\,s=n-q$, in\-te\-gra\-bi\-li\-ty of\/ $\,\dz\,$ implies
closedness of the in\-di\-vis\-i\-ble factor\/ $\,\zeta\hh$ of\/ 
$\,\my\,$ on every leaf of\/ $\,\dz$.
\end{lemma}
\begin{proof}Since $\,\dim\varSigma=s$, a volume $\,s$-form on $\,\varSigma\,$
chosen as in (\ref{vls}) is closed, leading to closedness of
$\,\hs\xi^1\nnh\nh\wedge\ldots\wedge\hs\xi^s$ in (\ref{vls}). With 
$\,\my=\xi^1\nnh\nh\wedge\ldots\wedge\hs\xi^s\nnh\nh\wedge\hs\zeta\,$ 
as in the lines following (\ref{vls}), 
$\,\xi^1\nnh\nh\wedge\ldots\wedge\hs\xi^s\nnh\nh\wedge\hh d\hh\zeta
=\hs d\my=0$. Thus, 
by Lemma~\ref{divpr}, $\,d\hh\zeta\,$ lies in the ideal generated by 
$\,\xi^1\nnh,\dots,\hs\xi^s\nnh$, that is, $\,d\hh\zeta=0\,$ on each leaf of
$\,\dz$.
\end{proof}
Whenever a nonzero differential $\,(n-2)$-form $\,\my\,$ with (\ref{cst}-b) 
on a manifold $\,M\,$ dimension  $\,n\,$  is {\it
in\-di\-vis\-i\-ble\/} in the sense -- cf.\ (\ref{ind}) --
of having the divisibility distribution 
$\,\dz\hs$ equal to $\,T\nh M\nh$, Lemma~\ref{duali} gives rise to a 
further invariant:
\begin{equation}\label{fin}
\text{\it the\ }\,2\hyp\text{\it form\ dual to\ }\,\my\text{\it,\ locally\
unique\ up\ to\ a\ sign.}
\end{equation}
More generally, for a differential $\,(n-2)$-form $\,\my\,$ 
on an $\,n$-di\-men\-sion\-al manifold $\,M$ satisfying (\ref{cst}-b) and 
having an in\-te\-gra\-ble 
divisibility distribution $\,\dz$, we can, by (\ref{cod}), apply the
last paragraph to any leaf $\,L\,$ of $\,\dz$, rather than $\,M\nh$, and --
instead of $\,\my$ itself -- to 
an in\-di\-vis\-i\-ble factor $\,\zeta\,$ of $\,\my\,$ on $\,L\,$ mentioned
in (\ref{idf}), obtaining 
\begin{equation}\label{tfd}
\begin{array}{l}
\mathrm{a\hs\ nondegenerate\ }\,\hs2\hyp\mathrm{form\
}\,\,\sy\,\hh\mathrm{\ on\ }\hs\,L\mathrm{,\hs\ 
unique\hh\ up\hh\ to\hh\ multiplications}\\
\mathrm{by\ nonzero\ constants,\nnh\ which\ is\ dual\ to\ an\
in\-di\-vis\-i\-ble\ factor\ of\ }\my\hh.
\end{array}
\end{equation}
\begin{lemma}\label{bivec}If a nonzero differential\/ $\,(n-2)$-form\/ 
$\,\my\,$ on an\/ $\,n$-di\-men\-sion\-al manifold\/ $\,M\,$ 
satisfies\/ {\rm(\ref{cst}-b)} and its divisibility distribution\/
$\,\dz\nh$, having the fibre dimension\/ $\,q$, is in\-te\-gra\-ble, the
bi\-vec\-tors on the leaves of\/ 
$\,\dz\hh$ reciprocal to the nondegenerate\/ $\,2$-forms\/ $\,\sy$
mentioned in\/ {\rm(\ref{tfd})} may be viewed, via an obvious push-for\-ward,
as forming a bi\-vec\-tor\/ $\,\beta\,$ defined locally in\/ $\,M\nnh$, and 
determined by\/ $\,\my$ uniquely up to multiplications by functions
constant along\/ $\,\dz$. Then
\begin{enumerate}
\item[(a)] locally in\/ $\,M\,$ there exist\/ $\,1$-forms\/ 
$\,\xi^{s+1}\nnh,\dots,\hs\xi^n$ with\/ 
$\,\sy=\hs\xi^{s+1}\nnh\nh\wedge\hs\xi^{s+2}\nh+\ldots
+\hs\xi\hh^{n-1}\nnh\nh\wedge\hs\xi\hh^n$ along each leaf, $\,q=n-s\,$ being 
even due to\/ {\rm(\ref{qis})}, 
\item[(b)] for any\/ $\,1$-forms\/ $\,\xi^1\nnh,\dots,\xi^s$ chosen as in\/
{\rm(\ref{vls})} and any\/ $\,\xi^{s+1}\nnh,\dots,\hs\xi^n$ as above,
$\,\xi^1\nnh,\dots,\hs\xi^n$ is a local trivialization 
of\/ $\,T^*\hskip-2.1ptM\,$ such that\/ 
$\,\my=\hs\omega\beta\,$ for the volume form\/
$\,\omega=\hs\xi^1\nnh\nh\wedge\ldots\wedge\hs\xi^n\nh$.
\end{enumerate}
\end{lemma}
\begin{proof}Being nondegenerate, $\,\sy\hh$ has the standard algebraic
type \cite[p.\,13]{bryant-chern-gardner-goldschmidt-griffiths}, and so 
$\,\xi^{s+1}\nnh,\dots,\hs\xi^n\nh\in T_{\!x}^*\nnh M\,$ required in (a) exist
at each point $\,x$, and may be augmented with $\,s\,$ additional $\,1$-forms
in $\,\txm\,$ to constitute a basis of $\,T_{\!x}^*\nnh M\nh$. Now (a) 
follows: $\,\xi^{s+1}\nnh,\dots,\hs\xi^n$ are just final portions of 
local trivializations of $\,T^*\hskip-2.1ptM\,$ 
dual to local trivializations of $\,T\nh M\,$ that are the smooth local 
sections of a $\,G$-prin\-ci\-pal bundle over $\,M\nh$, for a suitable 
matrix group $\,G\subseteq\mathrm{GL}\hh(n,\bbR)$. See, e.g., 
\cite[Sect.\,6]{derdzinski-piccione-terek}.

For (b), note that 
the definition (\ref{dfd}) of duality gives $\,\zeta\nh=\theta\hn\beta\,$ on
each leaf $\,L\,$ of $\,\dz$, with $\,M\nh,m,\my\,$ and $\,\omega\,$ 
replaced by $\,L,q/2,\zeta\hs$ and some suitable $\,\theta$.
By (\ref{vls}), $\,\my
=\hs\xi^1\nnh\nh\wedge\ldots\wedge\hs\xi^s\nnh\nh\wedge\hs[\beta\theta]\,$ 
and hence, as $\,q\,$ is even, (\ref{law}) yields 
$\,\omega\beta=\beta\hh\omega
=\beta[\hh\theta\wedge\hs\xi^1\nnh\nh\wedge\ldots\wedge\hs\xi^s]
=[\beta\theta]\wedge\hs\xi^1\nnh\nh\wedge\ldots\wedge\hs\xi^s\nh=\my$.
\end{proof}
The following obvious consequence of Lemma~\ref{dirpd} will be used in
Sect.\,\ref{dc}.
\begin{remark}\label{prdct}Given manifolds\/ 
$\,\varPi\hs$ and\/ $\,\varSigma$, a volume form\/ $\,\theta\,$ on\/ 
$\,\varSigma$, and an in\-di\-vis\-i\-ble closed 
differential\/ $\,r$-form\/ $\,\zeta$ on $\,\varPi\hs$ satisfying 
the condition\/
{\rm(\ref{cst}-b)}, let the symbols\/ $\,\theta\,$ and\/ $\,\zeta$ also stand
for the corresponding pull\-back forms on the product manifold\/
$\,M\nh=\varPi\times\nh\varSigma$. Then\/ 
$\,\my=\theta\wedge\zeta\,$ is a closed differential\/ $\,\px$-form\/ on\/
$\,M$ with the property\/ {\rm(\ref{cst}-b)}, for $\,\px=r+s\,$ and\/
$\,s=\dim\varSigma$, while 
the divisibility distribution\/
$\,\dz\hs$ of $\,\my\,$ is the factor distribution on\/ $\,M\,$ tangent 
to the\/ $\,\varPi\,$ factor manifold, and the restriction of\/ $\,\zeta\hs$
to $\,\dz$ is the in\-di\-vis\-i\-ble factor of\/ $\,\my$.
\end{remark}
\begin{lemma}\label{prjct}For a closed 
differential\/ $\,\px$-form\/ $\,\my\,$ with\/ {\rm(\ref{cst}-a)} on an\/
$\,n$-man\-i\-fold $\,M\nh$,
the kernel\/ $\,\zz\,$ is in\-te\-gra\-ble, and $\,\my\,$ is
pro\-ject\-a\-ble along\/ $\,\zz$, in the sense of Sect.\/~{\rm\ref{pr}},
onto a closed\/ $\,\px$-form on a local leaf space\/ $\,\varSigma$.
\end{lemma}
In fact, $\,\zz\,$ is in\-te\-gra\-ble by (\ref{dbr}). 
In local coordinates such that some of the coordinate fields
$\,\partial\nh_i\w$ span $\,\zz$, (\ref{dbr}) applied to $\,(\px+\hn1)$-tuples
of $\,\partial\nh_i\w$ implies  
constancy along $\,\zz\,$ of the components of $\,\my\,$ and closedness of
the projected $\,\px\hs$-form.
\begin{remark}\label{dboux}The Dar\-boux theorem with parameters. Let
$\,\zeta\hs$ be a con\-stant-rank section of $\,[\dz^*]^{\wedge2}$ for an 
in\-te\-gra\-ble distribution $\,\dz\,$ of fibre dimension $\,q\,$ on a 
manifold $\,M\nh$. If the restriction of $\,\zeta\hs$ to each leaf of
$\,\dz\hs$ is closed, then, locally in $\,M\nh$, and 
there exist functions $\,x^1\nh,\dots,x^q\nh$, constituting local coordinates
on each leaf of $\,\dz$, and such that 
$\,\zeta=\hs dx^1\nnh\wedge\hs dx^2\nh+\ldots
+\hs dx\hh^{r-1}\nnh\wedge\hs dx\hh^r\nh$, where
$\,r\hs$ is the (even) rank of $\,\zeta$. Namely, when $\,r\hn=q\,$ this is
\cite[Lemma 3.10]{bandyopadhyay-dacorogna-matveev-troyanov}. The general case
follows: the vector sub\-bun\-dle $\,\mathrm{Ker}\,\hs\zeta\hs$ of $\,\dz\hs$
is in\-te\-gra\-ble (Lemma~\ref{prjct}), and we may replace $\,M\,$ with
a local leaf space of $\,\mathrm{Ker}\,\hs\zeta$.
\end{remark}

\begin{remark}\label{volfr}It is well known --
see, e.g., \cite[Example\,1.6]{munoz-masque-pozo-coronado-rosado-maria} or 
\cite[Sect.\,11]{derdzinski-piccione-terek} -- that locally,
in any dimension $\,n$, any given volume form 
equals $\,\hs dx^1\nnh\nh\wedge\ldots\wedge\hs dx^n$ for suitable
coordinates $\,x^1\nnh,\dots,x^n\nnh$. This remains true in the 
hol\-o\-mor\-phic category, with the same
argument just cited from \cite{munoz-masque-pozo-coronado-rosado-maria} or 
\cite{derdzinski-piccione-terek}.
\end{remark}
\begin{remark}\label{dcomp}Let a de\-com\-pos\-able differential
$\,\px\hs$-form $\,\my\,$ on a manifold be algebraically constant (that is,
either identically zero, or nonzero everywhere). Then closedness of $\,\my\,$
is equivalent to its local constancy, as well as to its \igy. 
This is obvious from Remark~\ref{dvker}, Lemma~\ref{prjct}, and 
Remark~\ref{volfr} for $\,n=\px$.
\end{remark}
\begin{remark}\label{algct}The condition (\ref{cst}-b)
for a differential $\,(n-2)$-form $\,\my\,$ in dimension $\,n\,$ is equivalent 
to algebraic constancy of $\,\my$. Namely, in the lines following (\ref{vls}), 
$\,\my=\xi^1\nnh\nh\wedge\ldots\wedge\hs\xi^s\nnh\wedge\hs\zeta$, 
where $\,s\,$ is the co\-dimen\-sion of the divisibility distribution
$\,\dz\,$ and, by (\ref{cod}), the in\-di\-vis\-i\-ble factor $\,\zeta\hs$
restricted to $\,\dz\,$ has co\-de\-gree two. Being uniquely associated, via
(\ref{tfd}), with its dual $\,2$-form $\,\sy$, our $\,\zeta\hs$ is thus 
algebraically constant due to nondegeneracy of $\,\sy$.
\end{remark}

\section{Differential $\,3$-forms in dimension six}\label{dt}
\setcounter{equation}{0}
For a nonzero algebraically constant 
differential $\,3$-form on a $\,6$-di\-men\-sion\-al manifold $\,M\nh$,
each of the five cases of (\ref{six}) is realized, locally, by
\begin{equation}\label{smt}
\begin{array}{l}
\mathrm{smooth\ }\,1\hyp\mathrm{forms\ \
}\hs\xi^1\nnh,\dots,\hs\xi\hh^6\hs\hh\mathrm{\
trivializing\ \ }\,T^*\hskip-2.1ptM\nh,\\
\mathrm{dual\ to\ a\ local\ trivialization\ }e_1\w,\dots,e_6\w\nh\mathrm{\ of\
}\hs T\nh M.
\end{array}
\end{equation}
In fact, such $\,e_1\w,\dots,e_6\w$ are well known
\cite[Sect.\,6]{derdzinski-piccione-terek} to be precisely the smooth local
sections of a $\,G$-prin\-ci\-pal bundle over $\,M\nh$, for some 
matrix group $\,G\subseteq\mathrm{GL}\hh(n,\bbR)$. Consequently,
the invariants (\ref{fiv}) give rise, locally, to analogous smooth objects in
$\,M\nnh$, namely, an al\-most-com\-plex structure $\,J$, the
differential $\,3$-forms $\,\eta^+\nh,\eta^-$ and 
$\,\my(J\hs\cdot\,,\,\cdot\,,\,\cdot\,)$, 
the distributions
$\,\hz\,$ and $\,\hz^\pm\nh$,
\begin{equation}\label{vbi}
\mathrm{the\ vec\-tor}\hyp\mathrm{bun\-dle\ iso\-mor\-phism\
}\,\varTheta:\hz\to[\hz']^{\wedge2}\nh,
\end{equation}
the divisibility distribution $\,\dz\,$ of
$\,\my$, and finally -- if, in addition, $\,\dz\,$ is assumed to be
in\-te\-gra\-ble -- 
the in\-di\-vis\-i\-ble-fac\-tor $\,2$-form $\,\zeta\hs$ of $\,\my\,$ defined,
as in (\ref{idf}), along each leaf of $\,\dz$, and only unique on the 
leaf up to multiplications by nonzero constants. Note that, in
Sect.\,\ref{ds},
\begin{equation}\label{spn}
\hz\,\mathrm{\ is\ spanned\ by\ }\,e_2\w,e\hn_4\w,e_6\w\mathrm{,\ and\
}\,\hz^+\nnh,\,\hz^-\nh\mathrm{\ by\ }e\hn_4\w,e_5\w,e_6\w\mathrm{\ and\ 
}\,e_1\w,e_2\w,e_3\w\hh.
\end{equation}

\section{Local constancy and \igy\ of differential forms}\label{lc}
\setcounter{equation}{0}
The next result is immediate from\hn\ Theorems~\ref{tfsix}\hn\ and\hn\
\ref{lcint}, proved in \hbox{Sect.\,\ref{px}\hs--\hn\ref{pf}.}
\begin{theorem}\label{ififf}The local constancy of an
algebraically constant 
differential \hbox{$\,(n-2)$}-form on an\/ $\,n$-di\-men\-sion\-al manifold is 
equivalent to its being \ig.

This is also the case for\/ $\,3$-forms in dimension six.
\end{theorem}
The analog of Theorem~\ref{ififf} is known
\cite[Prop.\,D]{derdzinski-piccione-terek} to hold for 
differential $\,\px\hs$-forms in dimension $\,n$, where 
$\,\px\in\{0,1,2,n-1,n\}$. Thus, in dimensions $\,n\le6$, a 
differential form of any degree is locally constant if and only if it is \ig. 
However, \ig\ forms that are not locally constant exist in infinitely many
dimensions, starting from $\,7\,$ and $\,8$. See Theorem~\ref{injec}.

The objects $\,J,\hz,\hz^\pm\nh,\dz,\zeta\,$ in next theorem were described
in Sect.\,\ref{dt}.
\begin{theorem}\label{tfsix}The following three properties
of a nonzero algebraically constant differential\/ $\,3$-form\/ $\,\my\,$
on a $\,6$-di\-men\-sion\-al manifold 
are mutually equivalent.
\begin{enumerate}
\item[(i)] Local constancy.
\item[(ii)] Being \ig.
\item[(iii)] Closedness of\/ $\,\my$, coupled with
\begin{enumerate}
\item[(a)] in\-te\-gra\-bi\-li\-ty of the al\-most-com\-plex structure\/
$\,J\nh$, in case\/ {\rm(\ref{six}-a)},
\item[(b)] in\-te\-gra\-bi\-li\-ty of the distribution\/ $\,\hz$, when\/
{\rm(\ref{six}-b)} holds,
\item[(c)] in\-te\-gra\-bi\-li\-ty of both\/ 
$\,\hz^\pm$ under the assumption\/ {\rm(\ref{six}-c)},
\item[(d)] in\-te\-gra\-bi\-li\-ty of the divisibility distribution\/ 
$\,\dz\nh$, for\/ {\rm(\ref{six}-d)},
\item[(e)] no further condition in case\/ {\rm(\ref{six}-e)},
\end{enumerate}
\end{enumerate}
\end{theorem}
The notions of in\-di\-vis\-i\-ble factor and duality used below were
defined in Sect.\,\ref{si}.
\begin{theorem}\label{lcint}Given a nonzero algebraically constant
differential\/ $\,\px$-form\/ $\,\my$ on an\/ $\,n$-di\-men\-sion\-al
manifold\/ $\,M\nnh$, with the divisibility distribution\/ $\,\dz\,$ and the 
in\-di\-vis\-i\-ble factor\/ $\,\zeta$, 
the following two assumptions can be made about\/ $\,\my$.
\begin{enumerate}
\item[(a)] $\my\,$ is locally constant.
\item[(b)] $\my\,$ is \ig.
\end{enumerate}
The condition\/ {\rm(b)} always follows from {\rm(a)}, while\/ {\rm(b)}
implies that
\begin{enumerate}
\item[(i)] $\my\,$ is closed and the distribution\/ $\,\dz\,$ is
in\-te\-gra\-ble.
\end{enumerate}
If\/ $\,\px=n-2$, {\rm(b)} has a further consequence, namely,
\begin{enumerate}
\item[(ii)] along each leaf of\/ $\,\dz\nh$, the\/ $\,2$-form\/ $\,\sy$ dual
to\/ $\,\zeta\,$ is closed.
\end{enumerate}
Conversely, for\/ $\,\px=n-2$, {\rm(i)} 
and\/ {\rm(ii)} together imply\/ {\rm(a)}, and hence\/ {\rm(b)}.
\end{theorem}

\section{Proof of Theorem~\ref{tfsix}}\label{px}
\setcounter{equation}{0}
That (i)$\implies$(ii)$\implies$(iii) is obvious from (\ref{imp}) and
(\ref{iii}), since $\,J\,$ in (iii-a), due to its naturality, is 
$\,\nabla\nnh$-par\-al\-lel when a tor\-sion-free 
connection $\,\nabla\hs$ has $\,\nabla\nnh\my=0$. This last claim easily
follows from the New\-land\-er-Ni\-ren\-berg theorem, as pointed out
by various authors \cite[Sect.\,2.3]{clark-bruckheimer}, 
\cite[Definition\,2.2]{bolsinov-konyaev-matveev}.

We now proceed to show that (iii) implies (i) by establishing, 
for suitable local coordinates $\,x^1\nnh,\dots,x^n$ and
$\,\xi^1\nnh,\dots,\hs\xi\hh^6$ mentioned in (\ref{smt}), 
\begin{equation}\label{rmv}
\begin{array}{l}
\mathrm{each\ of\ the\ five\ equalities\ (\ref{six})\ with\ every\
}\,\xi\hs^i\mathrm{\ replaced\ by\ }\,dx^i\nh.
\end{array}
\end{equation}
First, in the case (\ref{six}-a), $\,\omega\,$ given by
(\ref{rlp}) is, by (\ref{tri}), a complex volume $\,(3,0)\,$ form on the 
complex manifold $\,M\nnh$, so that $\,\omega\,$ must be hol\-o\-mor\-phic,
due to closedness of its real part $\,\my\,$ and Remark~\ref{clsed}. 
The final clause of Remark~\ref{volfr} gives, locally, 
$\,\omega\hs=\hs dz^1\nnh\wedge\hs dz^2\nnh\wedge\hs dz^3$ in some 
hol\-o\-mor\-phic coordinates $\,z^1\nh,z^2\nh,z^3\nh$. The real coordinates 
$\,x^1\nnh,\dots,x^6$ with 
$\,(z^1\nh,z^2\nh,z^3)
=(x^4\nh+i\hh x^1\nh,x^6\nh+i\hh x^3\nh,x^2\nh+i\hh x\hh^5)\,$ now turn
(\ref{rlp}) into (\ref{rmv}).

Next, assume (iii) and (\ref{six}-b). 
The equality in (\ref{six}-b) still holds, according to Remark~\ref{replc}(a), 
with suitable $\,\hat\xi^2\nh,\hat\xi\hs^4\nh,\hat\xi\hh^6$ instead of
$\,\xi^2\nnh,\hs\xi\hs^4\nnh,\hs\xi\hh^6\nh$, if one replaces 
$\,\xi^1\nnh,\hs\xi^3\nnh,\hs\xi\hh^5$ with {\it any\/} local trivialization 
$\,\hat\xi^1\nh,\hat\xi^3\nh,\hat\xi\hh^5$ of $\,\hz'\nh$, the vector
sub\-bun\-dle of $\,T^*\hskip-2.1ptM\,$ polar to the 
distribution $\,\hz\,$ (and
$\,\hat\xi^1\nh,\dots,\hat\xi\hh^6$ will then still, locally, trivialize 
$\,T^*\hskip-2.1ptM$). 
Due to in\-te\-gra\-bi\-li\-ty of $\,\hz$,
we are therefore free to choose the triple
$\,(\xi^1\nnh,\hs\xi^3\nnh,\hs\xi\hh^5)\,$ 
in (\ref{six}-b) equal to $\,(dx^1\nh,\hs dx^3\nh,dx^5)$, with some functions 
$\,x^1\nh,x^3\nh,x^5$ constant along the leaves of $\,\hz$. For 
$\,e_1\w,\dots,e_6\w$ dual to 
$\,\xi^1\nnh,\dots,\hs\xi\hh^6$ as in (\ref{smt}), $\,e_2\w,e\hn_4\w,e_6\w$
form a local trivialization of $\,\hz$. Let the index ranges now be
$\,i,j=2,4,6$ and $\,k,l=1,3,5$. In (\ref{the}), 
$\,\varTheta e\hn_i\w=\my(e\hn_i\w,\,\cdot\,,\,\cdot\,)\,$ must thus be equal
to the corresponding 
$\,\hs\xi^k\nnh\wedge\hs\xi^l\hn=\hs dx^k\nnh\wedge\hs dx^l\nh$, and 
consequently 
annihilate $\,\hz$.
Closedness of all $\,\hs\xi^k\nh=\hs dx^k$ implies (see Remark~\ref{dufrm})
that all Lie brackets of $\,e_1\w,\dots,e_6\w$ are tangent to $\,\hz$.
From the last two sentences and (\ref{dbr}) with $\,d\my=0\,$ we now get
\[
0=-[d\my](e\hn_i\w,e\nh_j\w,\,\cdot\,,\,\cdot\,)
=\my([\hs e\hn_i\w,e\nh_j\w],\,\cdot\,,\,\cdot\,)
=\varTheta[\hs e\hn_i\w,e\nh_j\w]\hh,\qquad i,j=2,4,6\hh.
\]
Thus, due to the injectivity of $\,\varTheta\,$ in (\ref{vbi}), 
$\,e_2\w,e\hn_4\w,e_6\w$ commute with one another, and 
Lemma~\ref{coord} allows us to augment $\,x^1\nh,x^3\nh,x^5$ with three 
more functions so as to obtain, locally, a coordinate system
$\,x^1\nnh,y^2\nnh,x^3\nnh,y^4\nnh,x^5\nnh,y^6$ for which $\,e\hn_i\w$,
$\,i=2,4,6$, are the coordinate
vector fields $\,\partial\nh_i\w$. As
$\,[dy^j](e\hn_i\w)=[dy^j](\partial\nh_i\w)=\delta_i^j=\xi^j(e\hn_i\w)$,
each $\,\xi^j\nnh-\hs dy^j\nnh$, $\,j=2,4,6$, annihilates $\,\hz$, while 
$\,(\xi^1\nnh,\hs\xi^3\nnh,\hs\xi\hh^5)
=(dx^1\nh,\hs dx^3\nh,dx^5)$. Therefore, 
$\,\xi^2$ (or $\,\xi\hs^4\nh$, or $\,\xi\hh^6$) equals 
$\,dy^2\nh+\phi_5\w\hs dx^5$ (or 
$\,dy^4\nh+\phi_1\w\hh dx^1\nh$, 
or $\,dy^6\nh+\phi_3\w\hs dx^3$) plus a functional combination of 
$\,dx^1\nnh,\hs dx^3$ (or $\,dx^3\nnh,\hs dx^5$ or, respectively, 
$\,dx^1\nnh,\hs dx^5$), with some functions
$\,\phi_1\w,\phi_3\w,\phi_5\w$.
Substituting the expressions just obtained for
$\,\xi^2\nnh,\hs\xi\hs^4\nnh,\hs\xi\hh^6\nnh$, we rewrite (\ref{six}-b)
with $\,(\xi^1\nnh,\hs\xi^3\nnh,\hs\xi\hh^5)
=(dx^1\nh,\hs dx^3\nh,dx^5)\,$ as
\begin{equation}\label{rds}
\begin{array}{l}
\my\,=\,\hs\xi^1\nnh\nh\wedge\hs\xi^2\hn\nnh\wedge\hs\xi^3\nh
+\hs\xi^3\nnh\nh\wedge\hs\xi\hs^4\nh\nnh\wedge\hs\xi\hh^5\nh
+\hs\xi\hh^5\nnh\nh\wedge\hs\xi\hh^6\hn\nnh\wedge\hs\xi^1\hs
=\,dx^1\nnh\nh\wedge\hs dy^2\hn\nnh\wedge\hs dx^3\\
\hskip10pt+\hs\,\,dx^3\nnh\nh\wedge\hs dy^4\nh\nnh\wedge\hs dx^5\hs
+\hs\,dx^5\nnh\nh\wedge\hs dy^6\hn\nnh\wedge\hs dx^1\hs
-\,\phi\,dx^1\nnh\nh\wedge\hs dx^3\hn\nnh\wedge\hs dx^5,
\end{array}
\end{equation}
where $\,\phi=\phi_1\w+\phi_3\w+\phi_5\w$. Since $\,\my\,$ is closed,
(\ref{rds}) gives
$\,d\phi\hs\wedge\hs dx^1\nnh\nh\wedge\hs dx^3\hn\nnh\wedge\hs dx^5\nh=0$
and, 
by Lemma~\ref{divpr}, $\,\phi\,$ is a function of the variables
$\,x^1\nh,x^3\nh,x^5\nh$, thus equal -- see Remark~\ref{divrg} -- to
the divergence of some vector field $\,w=(w^1\nh,w^3\nh,w^5)$, with each 
$\,w^k$ depending only on $\,x^1\nh,x^3\nh,x^5\nh$. If we now set
$\,(x^2\nnh,\hs x^4\nnh,\hs x^6)
=(y^2\nh+w^5\nnh,\hs y^4\nh+w^1\nnh,\hs y^6\nh+w^3)$,
(\ref{rds}) becomes (\ref{rmv}) for the case (\ref{six}-b),
$\,x^1\nnh,\dots,x^6$ being local coordinates as 
linear independence of $\,dx^1\nnh,\dots,dx^6$ is immediate from 
the lines following (\ref{eql}). 

Suppose now that (iii) and (\ref{six}-c) hold. The sub\-bun\-dles $\,\hz^\pm$
of $\,T\nh M\nh$, being in\-te\-gra\-ble, are, locally, the factor
distributions of a Car\-te\-sian-prod\-uct decomposition of $\,M\nh$, and
so, by Lemma~\ref{plbcx}, $\,\eta^\pm$ are the pull\-backs to 
$\,M\,$ of some volume forms on the factor manifolds. Remark~\ref{volfr}
now gives, locally, 
$\,\eta^+\nh=\hs dx^1\nnh\wedge\hs dx^2\nnh\nh\wedge\hs dx^3$ and 
$\,\eta^-\nh=\hs dx^4\nnh\wedge\hs dx^5\nnh\nh\wedge\hs dx^6$ for some
local coordinates $\,x^1\nnh,x^2\nnh,x^3$ and 
$\,x^4\nnh,x^5\nnh,x^6$ in the factors, proving (\ref{rmv}) for (\ref{six}-c).

For (iii-d), Lemma~\ref{cllvs} and Remark~\ref{dboux} yield, locally, $\,\zeta
=\hs dx^2\nnh\wedge\hs dx^3\nh+\hs dx\hh^4\nnh\wedge\hs dx^5$ for suitable 
functions $\,x^2\nh,\dots,x^6$ constituting local coordinates on each leaf
of $\,\dz$, cf.\ (\ref{fiv}-d), while $\,\xi^1$ then becomes the volume
$\,1$-form on $\,\varSigma$, appearing in (\ref{vls}) with $\,s=1$.
One\hh-di\-men\-sion\-al\-i\-ty of $\,\varSigma$ now gives, locally,
$\,\xi^1\nh=\hs dx^1$
for some function $\,x^1$ with $\,dx^1\nh\ne0$, constant along the leaves of 
$\,\dz$, so that 
$\,\mu=\xi^1\nnh\nh\wedge\hs\zeta
=\hs dx^1\nnh\wedge\hs(dx^2\nnh\wedge\hs dx^3\nh
+\hs dx\hh^4\nnh\wedge\hs dx^5)\,$ in the resulting 
local coordinates $\,x^1\nnh,\dots,x^6\nh$, as required.

Finally, assuming (iii) and (\ref{six}-e), we get (\ref{rmv}) directly from
Remark~\ref{dcomp}.

This completes the proof of Theorem~\ref{tfsix}.

\section{Proof of Theorem~\ref{lcint}}\label{pf}
\setcounter{equation}{0}
By (\ref{imp}) and (\ref{iii}), (a)$\implies$(b)$\implies$(i). 
Deriving (ii) from (b) requires a more subtle argument, the
in\-di\-vis\-i\-ble factor $\,\zeta\hh$ of $\,\my\,$ being defined, along 
each leaf of $\,\dz$, only uniquely {\it up to multiplications by nonzero 
constants}. To this end, we assume (b). For tor\-sion-free $\,\nabla\hs$
with $\,\nabla\nnh\my=0$, the $\,1$-forms $\,\xi^1\nnh,\dots,\hs\xi^s$  
in (\ref{vls}) annihilate the $\,\nabla\nh$-par\-al\-lel distribution
$\,\dz$. By Lemma~\ref{annih}, $\,\xi\hh^i$ are all $\,\nabla\nh$-par\-al\-lel
along $\,\dz$, and hence so is
$\,\theta=\xi^1\nnh\nh\wedge\ldots\wedge\hs\xi^s\nnh$. Thus, along each leaf
$\,L\,$ of $\,\dz$, the restriction of $\,\zeta$ to $\,L\,$ being,
due to Lemma~\ref{divis}(c) uniquely determined by $\,\my\,$ (and our fixed
$\,\theta$), is \ig\ relative to the tor\-sion-free connection induced by
$\,\nabla\hs$ on the totally geodesic sub\-man\-i\-fold $\,L$. The same
then follows for the $\,2$-form dual to the restriction of $\,\zeta$, as the
latter determines the former up to a sign (Lemma~\ref{duali}). Now (ii) follows.

To prove the final clause of the theorem, suppose now that a nonzero
algebraically constant differential $\,(n-2)$-form
$\,\my\,$ on a manifold $\,M\,$ of dimension $\,n$ satisfies (i) and 
(ii). Being nondegenerate, the $\,2$-form $\,\sy\hs$ dual to $\,\zeta\hs$  is,
locally, 
a symplectic form on each leaf of $\,\dz$. The Dar\-boux theorem with 
parameters \cite[Lemma 3.10]{bandyopadhyay-dacorogna-matveev-troyanov}
allows us, locally, to write $\,\sy
=\hs dx^{s+1}\nnh\wedge\hs dx^{s+2}\nh+\ldots
+\hs dx\hh^{n-1}\nnh\wedge\hs dx^n$ 
for some functions $\,x^{s+1}\nh,\dots,x\hh^n$ with 
$\,dx^{s+1}\nh\nnh\wedge\ldots\hs\wedge dx\hh^n\nh\ne0$, where $\,s\,$ is the 
co\-dimen\-sion of $\,\dz$. 
Any (local) functions $\,x^1\nnh,\dots,x^s$ such that 
$\,dx^1\nh\wedge\ldots\hs\wedge dx^s\nh\ne0$ and the leaves of
$\,\dz\hs$ are the level sets of $\,(x^1\nh,\dots,x^s)$ give rise to
local coordinates $\,x^1\nnh,\dots,x^n\nh$. The last $\,n-s\,$ of the
corresponding coordinate vector fields $\,\partial\nh_i\w$ in $\,M\,$ serve
in the same capacity on leaves of $\,\dz$, as $\,x^1\nnh,\dots,x^s$
are constant along them, and so 
$\,\beta=-\hs\partial\nh_{s+1}\w\wedge\hs\partial\nh_{s+2}\w-\ldots
-\hs\partial\nh_{n-1}\w\wedge\hs\partial\nh_n\w$ for 
the bi\-vec\-tor $\,\beta\,$ in Lemma~\ref{bivec}. On the other hand, 
$\,dx^1\nnh,\dots,dx^s$ may serve as $\,\xi^1\nnh,\dots,\xi^s$ in 
(\ref{vls}), and Lemma~\ref{bivec}(b) applied to 
$\,(\xi^1\nnh,\dots,\xi^n\nh)=(dx^1\nnh,\dots,dx^n\nh)\,$ gives
$\,\my\hh=\hs\omega\beta\,$ for the volume $\,n$-form
$\,\omega=\hs dx^1\nh\nnh\wedge\ldots\hs\wedge dx^n\nh$. Thus, the 
components of $\,\my\,$ in the coordinates
$\,x^1\nnh,\dots,x^n$ are constant, as required.

\section{Iso\-tropy Lie algebras and connections}\label{il}
\setcounter{equation}{0}
In a real vector space $\,V$ of dimension $\,n$, any linear en\-do\-mor\-phism
$\,A\in\mathfrak{gl}\hh(V)$ acts on 
$\,[V\hn^*]^{\wedge\px}\nnh$, for $\,1\le\px\le n$, as the derivation 
$\,\my\mapsto A\my$, with $\,[A\my](v_1\w,\dots,v\hn_\px\w)$ equal to the sum
over $\,i=1,\dots,\px\,$ of the terms
$\,\my(\widetilde v_1\w,\dots,\widetilde v\hn_\px\w)$, where
$\,\widetilde v_i\w=Av_i\w$ and $\,\widetilde v\nh_j\w=v\nh_j\w$
if $\,j\ne i$. Clearly, for the obvious action of
$\,\mathrm{GL}\hh(V)\,$ on $\,[V\hn^*]^{\wedge\px}\nnh$,
\begin{equation}\label{gea}
\mathfrak{h}\,=\,\{A\in\mathfrak{gl}\hh(V):A\my=0\}\,\mathrm{\ is\ the\
iso\-tropy\ Lie\ algebra\ of\ }\,\my\hh.
\end{equation}
Of particular interest to us is the case where
\begin{equation}\label{cse}
\begin{array}{l}
\mathrm{In\,\ 
(\ref{gea}),\ all\ }\hh\,A\hs\in\hs\mathfrak{h}\hh\,\mathrm{\ are\
skew}\hyp\mathrm{ad\-joint\ relative}\\
\mathrm{to\ some\ pseu\-do\hs}\hyp\mathrm{Eu\-clid\-e\-an\ inner\ product\
in\ \hn}\,V\nnh.
\end{array}
\end{equation}

Given a manifold $\,M\nh$, consider the in\-fi\-nite-di\-men\-sion\-al
af\-fine space $\,\mathcal{C}(M)\,$ of all tor\-sion-free connections on
$\,M\,$ and, for any fixed differential $\,\px\hs$-form $\,\my\,$ on $\,M\nh$,
\begin{equation}\label{aff}
\mathrm{the\ affine\ mapping\
}\,\,\mathcal{C}(M)\ni\,\nabla\,\mapsto\,\nabla\nnh\my\hh,
\end{equation}
valued in $\,(0,\px\hn+1)\,$ tensor fields, skew-sym\-met\-ric in the last
$\,\px\,$ arguments.
\begin{lemma}\label{atmst}Let an algebraically constant differential\/
$\,3$-form\/ $\,\my\,$ on an $\,n$-di\-men\-sion\-al manifold\/ $\,M\hs$
have an algebraic type that satisfies\/ {\rm(\ref{cse})}.
Then the mapping\/ {\rm(\ref{aff})} is injective and, 
even locally, there exists at most one tor\-sion-free connection\/
$\,\nabla\nh$ on\/ $\hs\,M\,$ with\/ $\,\nabla\nnh\my=0$.
\end{lemma}
\begin{proof}Suppose that a $\,(1,2)\,$ tensor field $\,B\,$ is the difference
between two tor\-sion-free connections on $\,M\,$ assigining to $\,\my\,$ the
same covariant derivative. Thus, in local coordinates, 
$\,B_{i\hn j}^s\my_{skq}\w+B_{ik}^s\my_{jsq}\w+B_{iq}^s\my_{jks}\w=0$, 
that is, if a vector $\,v\,$ is tangent to $\,M\,$ at a point $\,x$, the 
$\,(1,1)\,$ tensor $\,B_x\w(v,\,\cdot\,)\,$ equals, by (\ref{gea}), the value 
at $\,x\,$ of some element $\,A\,$ of the iso\-tropy Lie algebra 
$\,\mathfrak{h}\subseteq\mathfrak{gl}\hh(\txm)\,$ of $\,\my_x\w$. For a 
pseu\-\hbox{do\hs-}Eu\-clid\-e\-an inner product $\,g\,$ in $\,\txm\hs$
chosen as 
in (\ref{cse}), $\,g_{ks}\w B_{i\hn j}^s(x)\,$ is thus symmetric in $\,i,j\,$ and
skew-sym\-met\-ric in $\,j,k$, so that it must vanish, proving the injectivity
claim. The remaining assertion is now obvious, with the `even lo\-cal\-ly'
part immediate as $\,M\hs$ may be replaced by any open sub\-man\-i\-fold.
\end{proof}
Following Joyce \cite[Sect.\,2.2--2.3]{joyce}, for a Euclidean vector space
$\,V$ of dimension $\,n\,$
and a basis $\,\xi^1\nnh,\dots,\hs\xi^n$ of $\,V\hn^*$ dual to an or\-tho\-nor\-mal
basis 
of $\,V\nh$, we write $\,\xi\hs^{i\hn j\dots k}$ for 
$\,\xi\hs^i\nnh\nh\wedge\hs\xi^j\nnh\wedge\ldots\wedge\hs\xi\hh^k\nnh$, and 
consider, when $\,n=7$, the exterior $\,3$-form
\begin{equation}\label{sev}
\my\,=\,\hs\xi^{123}\nnh+\hh\xi^{145}\nnh+\hh\xi^{167}\nnh+\hh\xi^{246}\nnh-\hh\xi^{257}\nnh-\hh\xi^{347}\nnh-\hh\xi^{356}\nh,
\end{equation}
while, if $\,n=8$, we use the same symbol for the exterior $\,4$-form 
\begin{equation}\label{egh}
\begin{array}{l}
\my\,=\,\hs\xi^{1234}\nnh+\hh\xi^{1256}\nnh+\hh\xi^{1278}\nnh
+\hh\xi^{1357}\nnh
-\hh\xi^{1368}\nnh-\hh\xi^{1458}\nnh-\hh\xi^{1467}\\
\hskip10pt-\,\,\,\hs\xi^{2358}\nnh
-\hh\xi^{2367}\nnh-\hh\xi^{2457}\nnh+\hh\xi^{2468}\nnh+\hh\xi^{3456}\nnh
+\hh\xi^{3478}\nnh+\hh\xi^{5678}\nh.
\end{array}
\end{equation}
In both cases, the isotropy group of $\,\my\,$ in $\,\mathrm{GL}\hh(V)$, 
iso\-mor\-phic to $\,G_2\w$ or, respectively, $\,\mathrm{Spin}\hh(7)$, 
preserves the inner product \cite[Sect.\,2.2--2.3]{joyce}. Thus,
\begin{equation}\label{sat}
\mathrm{both\ (\ref{sev})\ and\ (\ref{egh})\ have\ the\ property\ (\ref{cse}).}
\end{equation}

\section{The Car\-tan $\,3$-forms of simple Lie algebras}\label{ct}
\setcounter{equation}{0}
Our convention about the sign of the curvature tensor $\,R\,$ of a connection
$\,\nabla$ on a manifold $\,M\,$ is such that, for vector fields $\,u,v,w$,
\begin{equation}\label{rvw}
\begin{array}{l}
R\hs(v,w)u\,=\,\naw\nav u\,-\,
\nav\naw u\,+\,\nabla\!_{\nh[v,w]}\w u
\end{array}
\end{equation}
When $\,\nabla$ is the Le\-vi-Ci\-vi\-ta connection of a 
pseu\-do\hs-Riem\-ann\-i\-an metric $\,g\,$ on $\,M\nh$, we may treat $\,R\,$
as a vec\-tor-bun\-dle mor\-phism
\begin{equation}\label{vbm}
\mathrm{a)}\hskip6ptR:[T\hskip.2pt^*\hskip-2ptM]^{\otimes2}\nnh
\to[T\hskip.2pt^*\hskip-2ptM]^{\otimes2}\nh,\quad
\mathrm{b)}\hskip6pt\mathrm{leaving\
}\,[T\hskip.2pt^*\hskip-2ptM]^{\odot2}\nnh\mathrm{\ and\ 
}\,[T\hskip.2pt^*\hskip-2ptM]^{\wedge2}\nnh\mathrm{\ invariant,}
\end{equation}
so that it acts on arbitrary $\,(0,2)\,$ tensor fields $\,b\,$
by
\begin{equation}\label{arb}
[\nh Rb\hs]_{i\hn j}\w\,=\hs\,R_{ipjq}\w b\hs^{pq}\nh,
\end{equation}
with index raising and
lowering via $\,g$, and summation over repeated indices. See
\cite[Sect.\,1.114,\,1.131]{besse}. This action not only 
preserves (skew\hn)\hh sym\-metry of $\,b$, but also --  due to the first
Bianchi identity -- amounts to the usual formula
\begin{equation}\label{std}
2[R\zeta]_{i\hn j}\w\,=\,R_{i\hn jpq}\w\zeta\hh^{pq}
\end{equation}
if $\,b=\zeta\,$ happens to be 
skew-sym\-met\-ric (a $\,2$-form).

Given a connected Lie group $\,G$, with the Lie algebra $\,\mathfrak{g}\,$ of 
left-in\-var\-i\-ant vector fields, we treat the Kil\-ling form $\,g$, 
the Car\-tan $\,3$-form $\,\gamma$, characterized by
$\,g(u,v)=\hs\mathrm{tr}\,[(\mathrm{Ad}\hskip1.3ptu)\hs\mathrm{Ad}\hskip1.3ptv]\,$ and
$\,\gamma(u,v,w)=g([u,v],w)\,$ whenever $\,u,v,w\in\mathfrak{g}$, and
the Lie bracket $\,C:\mathfrak{g}\times\mathfrak{g}\to\mathfrak{g}$, as 
left-in\-var\-i\-ant tensor fields of types $\,(0,2)$, $\,(0,3)\,$ and
$\,(1,2)\,$ on $\,G$, which then makes them -- see below -- also
bi-in\-var\-i\-ant. Setting
\begin{equation}\label{stb}
\nav w\,=\,[v,w]/2\,\mathrm{\ \ for\ \ }\,v,w\in\mathfrak{g}\hh,
\end{equation}
we define the {\it standard 
bi-in\-var\-i\-ant tor\-sion-free connection\/} $\,\nabla$ on $\,G$. By 
(\ref{rvw}), $\,\nabla$ has the $\,\nabla\nh$-par\-al\-lel curvature tensor
$\,R\,$ with
\begin{equation}\label{frv}
4R\hs(v,w)u\,=\,[[v,w],u]\,\mathrm{\ for\ }\,u,v,w\in\mathfrak{g}\hh.
\end{equation}
Bi-in\-var\-i\-ance of $\,C\,$ and $\hs\nabla$ trivially follows from the
dif\-feo\-mor\-phic in\-var\-i\-ance of the Lie bracket, while that of
$\,g\,$ (and, consequently, $\,\gamma$), as well as the fact that 
$\,g,\gamma,C$ and $\,R\,$ are all $\,\nabla\nh$-par\-al\-lel, is due to the
Ja\-co\-bi identity.

In any local coordinates $\,x^1\nnh,\dots,x^n$ for $\,G$, unrelated to the
Lie-group structure, our tensor fields have the component functions
$\,g_{i\hn j}\w$, 
$\,\gamma_{i\hn jk}\w$ and $\,C_{i\hn j}^k$, with 
\begin{equation}\label{gij}
\mathrm{a)}\hskip6ptg_{i\hn j}\w=C_{ir}^sC_{js}^r\hh,\qquad
\mathrm{b)}\hskip6pt\gamma_{i\hn jk}\w=C_{i\hn j}^rg_{rk}\w\hh,\qquad
\mathrm{c)}\hskip6pt4R_{i\hn jk}\w{}^q\nh=C_{i\hn j}^rC_{rk}^q\hn.
\end{equation}
In the case where $\,\mathfrak{g}\,$ is sem\-i\-simple, which allows
us to use $\,g$-in\-dex raising, $\,g\,$ is clearly a locally symmetric 
pseu\-do\hs-Riem\-ann\-i\-an Ein\-stein metric with the Le\-vi-Ci\-vi\-ta
connection
$\,\nabla\nh$, while (\ref{gij}-a) and (\ref{gij}-c) can be rewritten as
\begin{equation}\label{gma}
\gamma_{ipq}\w\gamma\hh^{pqj}\hs=\,-\delta_i^j\hh,\qquad
4R_{i\hn jkq}\w=C_{i\hn j}^r\gamma_{rkq}\w=g^{rs}\gamma_{i\hn jr}\w\gamma_{kqs}\w\hh.
\end{equation}
The dimension restriction in the next lemma amount to requiring that
$\,\mathfrak{g}\,$ not be iso\-mor\-phic to 
$\,\mathfrak{sl}\hh(2,\bbR)$, $\,\mathfrak{sl}\hh(2,\bbC)$, or 
$\,\mathfrak{su}\hh(2)=\mathfrak{so}\hh(3)$.
\begin{lemma}\label{isotr}For any simple Lie algebra\/ $\,\mathfrak{g}\,$ of
real dimension\/ $\,n\ge8$, the iso\-tropy Lie algebra\/ 
$\,\mathfrak{h}\subseteq\mathfrak{gl}\hh(\mathfrak{g})$, with\/
{\rm(\ref{gea})}, of the Car\-tan\/ 
$\,3$-form\/ $\,\gamma\in[\mathfrak{g}^*]^{\wedge3}\nnh$ equals the image\/
$\,\{\mathrm{Ad}\hskip1.3ptv:v\in\mathfrak{g}\}\,$ of the\/ $\,\mathrm{Ad}\,$
representation.

Consequently, {\rm(\ref{cse})} holds for\/ $\,V\nh=\mathfrak{g}\,$ and\/
$\,\my=\gamma$, and on any connected simple Lie group\/ $\,G\hs$
of dimension\/ $\,n\ge8\,$ 
the standard bi-in\-var\-i\-ant tor\-sion-free 
connection\/ $\,\nabla\hh$ given by\/ {\rm(\ref{stb})} is, 
even locally, the only tor\-sion-free connection on\/ $\,G\hs$ that makes the 
Car\-tan\/ $\,3$-form\/ $\,\gamma\,$ parallel.
\end{lemma}
\begin{proof}Let us fix $\,A\in\mathfrak{h}\,$ and identify
$\,\mathfrak{g}\,$ with the Lie algebra of
left-in\-var\-i\-ant vector fields on a connected Lie group $\,G$. Then, 
in local coordinates $\,x^1\nnh,\dots,x^n$ as above, 
$\,A\,$ treated as a left-in\-var\-i\-ant tensor field of
type $\,(1,1)\,$ on $\,G\,$ satisfies the relation 
$\,A_i^s\gamma_{sjk}\w+A_j^s\gamma_{isk}\w+A_k^s\gamma_{i\hn js}\w=0\,$ 
which, contracted against $\,\gamma\hh^{jkp}\nh$, gives, due to (\ref{gma}),
$\,A_j^i=8A_p^qR^{\hh ip}{}\nnh_{jq}\w$. In other words, 
$\,a^*\nh=8Ra\,$ for the $\,(0,2)\,$ tensors $\,a,a^*$ 
at any point $\,x\in G\hs$ defined by
$\,a_{i\hn j}\w=A_i^sg_{sj}\w$ and $\,a^*_{i\hn j}=a_{ji}\w$. 
The operator (\ref{vbm}-a) for $\,R\,$ in (\ref{frv}) and $\,M\nh=\hs G\hs$ 
commutes, by (\ref{vbm}-b), with $\,a\mapsto a^*\nh$, and so
\begin{equation}\label{eig}
8Ra^\pm\hh=\,\hs\pm a^\pm\mathrm{\ \ for\ }\,a^\pm\hs=\,(a\pm a^*)/2\hh.
\end{equation}
The spectrum of (\ref{vbm}-a) for $\,R\,$ in (\ref{frv}) is completely
understood for all simple Lie groups, via an easy argument in
\cite{derdzinski-gal} for the restriction 
$\,R:[\mathfrak{g}^*]^{\wedge2}\nnh
\to[\mathfrak{g}^*]^{\wedge2}\nnh$, valid in the general
sem\-i\-simple case, and, for the other restriction, 
$\,R:[\mathfrak{g}^*]^{\odot2}\nnh
\to[\mathfrak{g}^*]^{\odot2}\nnh$, due to a result of Meyberg
\cite{meyberg}, also presented in \cite[the Appendix]{derdzinski-gal}. 

Namely, the operator $\,T\,$ in \cite[formula (2.6)]{derdzinski-gal} 
is, by (\ref{arb}) and (\ref{gma}), equal to $\,-8R$, for our
$\,R\,$ in (\ref{vbm}-a) and (\ref{frv}), and so, according to 
\cite[Lemma 2.1(c)--(d)]{derdzinski-gal}, 
$\,R:[\mathfrak{g}^*]^{\wedge2}\nnh
\to[\mathfrak{g}^*]^{\wedge2}$ is di\-ag\-o\-nal\-iz\-able
with the eigen\-values $\,0\,$ and $\,-\nh1/\hh8$, 
while its eigen\-space
for the eigenvalue $\,-\nh1/\hh8\,$ is
$\,\{\gamma(v,\,\cdot\,,\,\cdot\,):v\in\mathfrak{g}\}$. 

On the other hand, under our assumption about the dimension of 
$\,\mathfrak{g}$, \cite[Remark 4.5]{derdzinski-gal} implies that
$\,1/\hh8\,$ is {\it not\/} an eigen\-value of
$\,R:[\mathfrak{g}^*]^{\odot2}\nnh\to[\mathfrak{g}^*]^{\odot2}\nnh$. Note
that, according to \cite[Lemma 2.1(b)]{derdzinski-gal},
$\,\varOmega\,$ in \cite[Remark 4.5]{derdzinski-gal} equals
$\,T\nh$, and hence our $\,-8R$. Thus, by (\ref{eig}), $\,a=a^-$ lies, at each
point, in $\,\{\gamma(v,\,\cdot\,,\,\cdot\,):v\in\mathfrak{g}\}$, 
which is the eigen\-space just mentioned, that
is, $\,A\in\{\mathrm{Ad}\hskip1.3ptv:v\in\mathfrak{g}\}$. 
As the opposite inclusion 
$\,\{\mathrm{Ad}\hskip1.3ptv:v\in\mathfrak{g}\}\subseteq\mathfrak{h}\,$
amounts to the aforementioned bi-in\-var\-i\-ance of $\,\gamma$, the first
part of the lemma follows, while the final clause is then 
obvious from Lemma~\ref{atmst}, since\/ $\,\nabla\gamma=0\,$ the 
according to the lines following (\ref{frv}).
\end{proof}

\section{\Igy\ without local constancy}\label{pw}
\setcounter{equation}{0}
As mentioned in the Introduction, the converse of the first implication in
(\ref{imp}) for $\,\px\hs$-forms in dimension $\,n\,$ fails in general, unless
$\,\px\in\{0,1,2,n-2,n-1,n\}$. Here are some explicit examples.
\begin{theorem}\label{injec}If an exterior\/ $\,\px\hs$-form in dimension\/
$\,n\,$ has the algebraic type
\begin{enumerate}
\item[(a)] of\/ {\rm(\ref{sev})} with\/ $\,(n,\px)=(7,3)$, or\/
{\rm(\ref{egh})}
for\/ $\,(n,\px)=(8,4)$, or
\item[(b)] the Car\-tan\/ $\,3$-form 
of any simple Lie algebra\/ $\,\mathfrak{g}\,$ with\/
$\,\dimr\mathfrak{g}=n\ge8$,
\end{enumerate}
then it can be realized as a 
differential\/ 
$\,\px\hs$-form\/ $\,\my\,$ on an\/ $\,n$-di\-men\-sion\-al manifold, 
so as to be \ig, but not locally constant on any open
sub\-man\-i\-fold.

Specifically,
for\/ {\rm(a)} we may choose\/ $\,\my\,$ to
be a specific parallel\/ $\,\px\hs$-form on a compact simply connected
Riemannian manifold of dimension\/ $\,n\in\{7,8\}\,$ with the holonomy group\/ 
$\,G\nh_2\w$ or\/ $\,\mathrm{Spin}\hh(7)\,$ while, for\/ {\rm(b)}, we let\/
$\,\my\,$ be the Car\-tan\/ $\,3$-form 
on a connected 
Lie group having\/ $\,\mathfrak{g}\,$ as the Lie algebra
of left-in\-var\-i\-ant vector fields.
\end{theorem}
\begin{proof}Our choice for (a) is possible according to Joyce
\cite[Sect.\,2.2--2.3]{joyce}, with the resulting Le\-vi-Ci\-vi\-ta connection
$\,\nabla$ that must be Ric\-ci-flat, and hence real-an\-a\-lyt\-ic
\cite{deturck-kazdan}. By (\ref{sat}) and Lemma~\ref{atmst}, 
$\,\nabla\hs$ is, even locally, the only tor\-sion-free connection with
$\,\nabla\nnh\my=0$. Having the holonomy group $\,G_2\w$ or 
$\,\mathrm{Spin}\hh(7)$, it is not flat either on $\,M\nh$, or
-- due to an\-a\-lyt\-ic\-i\-ty -- on any open sub\-man\-i\-fold.

For (b), our assertion is in turn obvious from the uniqueness assertion in the
final clause of Lemma~\ref{isotr} 
combined with (\ref{ifi}), since $\,\nabla$ in Lemma~\ref{isotr} is --
according to the lines preceding (\ref{gma}) -- 
the Le\-vi-Ci\-vi\-ta connection of the pseu\-\hbox{do\hs-} Riem\-ann\-i\-an
Ein\-stein metric $\,g\,$ (the Kil\-ling form), which is not flat
on any open sub\-man\-i\-fold, as it has,
by (\ref{gma}), the nonzero parallel Ric\-ci tensor $\,-g/4$.
\end{proof}
\begin{remark}\label{smplg}Any real simple Lie algebra $\,\mathfrak{g}\,$ is
either a real form of a complex simple Lie algebra $\,\mathfrak{h}$, or the
result of treating some such $\,\mathfrak{h}\,$ as real. See, e.g.,
\cite[Lemma 4 on p.\,173]{hausner-schwartz}. 
According to \cite[Theorem 4.1]{derdzinski-gal}, the curvature operator 
$\,R:[\mathfrak{g}^*]^{\odot2}\nnh\to[\mathfrak{g}^*]^{\odot2}$ in
Sect.\,\ref{ct} has the same nonzero eigen\-values as its analog for
$\,\mathfrak{h}$. It also behaves additively under the di\-rect-sum
operation applied to Lie algebras. This generalizes the final clause of
Lemma~\ref{isotr} and Theorem~\ref{injec}(b) to the case of arbitrary 
sem\-i\-simple Lie algebras without ideals of dimensions $\,3\,$ or $\,6$.
\end{remark}

\section{Logical independence in Theorems~\ref{tfsix} and
\ref{lcint}}\label{li}
\setcounter{equation}{0}
Closedness of an algebraically constant differential $\,\px\hs$-form $\,\my\,$
in dimension $\,n$ is known {\it not\/} to imply 
in\-te\-gra\-bi\-li\-ty of its  divisibility distribution $\,\dz\hh$ except
when $\,\px\in\{0,1,2,n-1,n\}\hn$: counterexamples in
\cite[Sect.\,12]{derdzinski-piccione-terek} 
realize all dimensions $\,n\ge5$ and all $\,\px\,$ with
$\,2<\px<n-1$.

It is thus natural to ask if other parts of items (iii) in
Theorem~\ref{tfsix} and (i), (ii) in Theorem~\ref{lcint} are similarly free
of redundancy. Theorem~\ref{tfsix} gives rise to three questions of this kind 
-- whether closedness of $\,\my\,$ implies any of (a), (b), (c) -- and 
Theorem~\ref{lcint} to one more: does (ii) follow from (i)?

This section answers the first three questions in the negative with the
aid of the following examples. They use $\,e_1\w,\dots,e_6\w$, dual to
$\,\xi^1\nnh,\dots,\hs\xi\hh^6$ as in
(\ref{smt}), chosen to form a basis of the Lie algebra 
of left-in\-var\-i\-ant vector fields on a Lie group, the only nonzero 
Lie brackets being those algebraically related to
\begin{equation}\label{rel}
\begin{array}{l}
[\hs e_1\w,e_2\w]=e_5\w\hh,\qquad[\hs e_1\w,e_3\w]=e_6\w\qquad\mathrm{\ for\
\ (a),}\\
{}[\hs e_2\w,e\hn_4\w]=e_1\w\hh,\qquad[\hs e_6\w,e_2\w]=e_3\w\qquad\mathrm{\ for\
\ (b),}\\
{}[\hs e_1\w,e_2\w]=e_6\w\hh,\qquad[\hs e_5\w,e\hn_4\w]=e_3\w\qquad\mathrm{\
for\ \ (c).} 
\end{array}
\end{equation}
The Ja\-co\-bi identity is obvious, since all brackets lie in the center. The
only nonzero components of $\,\my\,$ thus are, cf.\ (\ref{six}), up to obvious
consequences of skew-sym\-me\-try, 
$\,\my_{123}\w=\my_{345}\w=\my_{561}\w=\my_{246}\w=1\,$ for (a),
$\,\my_{123}\w=\my_{345}\w=\my_{561}\w=1\,$ for (b), 
$\,\my_{123}\w=\my_{456}\w=1\,$ for (c). 
By (\ref{dbr}), with the same convention, the only pos\-si\-bly-non\-ze\-ro
components of $\,\zeta=d\my\,$ are $\,\zeta_{i\hn jkl}\w$ such that 
$\,[\hs e\hn_i\w,e\nh_j\w]\ne0$, Explicitly, in case (a): 
$\,\zeta_{12\hh kl}\w=0\,$ and $\,\zeta_{13\hs i\hn j}\w=0$ for $\,kl$, or
$\,i\hn j$, ranging over $\,34,35,36,45,46,56\,$ or, respectively,
$\,45,46,56$. Similarly, for (b): $\,\zeta_{24kl}\w=0\,$ with 
$\,kl=13,15,16,35,36,56$, and $\,\zeta_{26\hh i\hn j}\w=0\,$ for $\,ij=13,15,35$.
Finally, in case (c), $\,\zeta_{12\hh kl}\w=0\,$ and
$\,\zeta_{45i\hn j}\w=0\,$ with $\,kl$, or 
$\,ij$, ranging over $\,34,35,36,45,46,56\,$ or, respectively,
$\,13,16,23,26,36$. Thus, $\,d\my=0\,$ in all cases, while (\ref{rel}) and
(\ref{spn}) show that neither $\,\hz$, for (b), nor either of $\,\hz^\pm\nh$,
for (c), is in\-te\-gra\-ble. Finally, the Nijen\-huis tensor $\,N\nh$ of
$\,J\,$ sends vector fields $\,v,w\,$ to 
$\,N(v,w)=J[J v,w]+J[v,J w]-[J v,J w]+[v,w]$, so that
$\,N(e_1\w,e_2\w)=e_5\w$ by (\ref{rel}) and (\ref{acs}), proving
non-in\-te\-gra\-bi\-li\-ty of $\,J\,$ in case (a).

For the remaining (fourth) redundancy question raised in the initial
paragraph of this section, a negative answer is provided by
Remark~\ref{redun} below.

\section{Duality and closedness}\label{dc}
\setcounter{equation}{0}
Given a nondegenerate 2-form $\,\sy\hs$ on an $\,n$-di\-men\-sion\-al manifold
$\,M\nh$, with $\,n=2m\ge2\,$ even, we use (\ref{dfd}) to define
its {\it dual\/ $\,(n-2)$-form\/} $\,\my$. Note that
\begin{equation}\label{cls}
\begin{array}{l}
\mathrm{if\ }\hs\hs\sy\hh\mathrm{\ is\ closed,\nnh\ or\ locally\ constant,\nnh\
or\,\ par}\hyp\\
\mathrm{allel,\ then\ so\ is,\ respectively,\ its\ dual\ }\,\,\my\hh,
\end{array}
\end{equation}
the locally-constant and \ig\ cases obvious as 
$\,\my\,$ and $\,\sy\hs$ arise from each other via explicit constructions
(Lemma~\ref{duali}). For the same reason, conversely,
\begin{equation}\label{cnv}
\mathrm{\igy\ or\ local\ constancy\ of\ the\ dual\ of\ 
}\,\sy\hs\mathrm{\ implies\ the\ same\ for\ }\,\sy\hh.
\end{equation}
Now (\ref{cls}) follows: closedness of $\,\sy\hs$ is equivalent to its local
constancy due to the Dar\-boux theorem, which gives
$\,\sy=\hs dx^1\nnh\wedge\hs dx^2\nh+\ldots+\hs dx\hh^{n-1}\nnh\wedge\hs dx^n$
in suitable local coordinates $\,x^1\nnh,\dots,x^n\nnh$. Thus, by (\ref{mds}),
the dual $\,\my\,$ of $\,\sy\hs$ has constant component functions in these 
coordinates, and is consequently closed.

The symbol $\,\widehat{\,\,}\,$ means `delete' in the following
theorem, which shows that, in contrast with (\ref{cnv}), the converse of
the ``closed'' case of (\ref{cls}) is generally
false in all (necessarily even) dimensions $\,n\ge6$. 
Note that for $\,n=2\,$ and $\,n=4\,$ the converse is true, since then
$\,\my=-\nnh1\,$ and $\,\my=-\hn\sy\hs$ by (\ref{led}).

The index ranges below are $\,i,j=1,\dots,m\,$ and $\,k=1,\dots,2m$.
\begin{theorem}\label{duncl}Given an integer\/ $\,m\ge2\,$ and positive
functions\/ $\,\phi\hn_i\w$ of
the\/ $\,2m$ real variables\/ $\,x^k\nh$, let\/ 
$\,\chi_i\w=\hs dx^{2i-1}\nnh\wedge dx^{2i}$ and\/  
$\,\sy\nnh_i\w=\phi_{\nh i}^{1-m}\phi\hh\chi_i\w$, where\/ 
$\,\phi=\phi\hn_1\w\ldots\phi\hn_m\w$. The\/ $\,2$-form\/ $\,\sy\hs$
and\/ $\,(2m-2)$-form $\,\,\,\zeta\,$ defined by\/ 
$\,\sy=\sy\nh_1\w+\ldots+\sy\nnh_m\w$ 
and\/ $\,\zeta=\zeta_1\w\nh+\ldots+\hs\zeta_m\w$, with\/ 
$\,\zeta_i\w
=-\sy\nh_1\w\wedge\ldots\wedge\widehat{\sy}\nnh_i\w\wedge\ldots
\wedge\sy\nnh_m\w$, then have the following properties.
\begin{enumerate}
\item[(a)] $\zeta\,$ is algebraically constant and in\-di\-vis\-i\-ble.
\item[(b)] $\sy\,$ and\/ $\,\zeta\hs$ are dual to each other.
\item[(c)] $\zeta\hs$ is closed if and only if, with no summation, 
$\,\partial\hn_{2i-1}\w\phi\hn_i\w=\partial\hn_{2i}\w\phi\hn_i\w=0$.
\item[(d)] $\sy\,$ is closed if and only if\/
$\,\partial\nh_k\w[\hh\phi_{\nh i}^{1-m}\phi]=0\,$ whenever\/
$\,k\notin\{2i-1,2i\}$, which is in turn equivalent to having\/
$\,\phi\hn_i\w=\hs\rho_1\w\nnh\ldots\hs\widehat\rho_i\w\nh\ldots\hs\rho_m\w$
for all\/ $\,i$, where each\/ $\,\rho\hn_j\w$ is a function of\/
$\,x^{2j-1}\nnh\nh$ and $\,x^{2j}\nh$.
\end{enumerate}
\end{theorem}
\begin{proof}
Algebraic constancy of $\,\zeta\hh$ follows as 
$\,\sy\nnh_i\w=\xi^{2i-1}\nnh\wedge\hs\xi^{2i}$ with 
$\,\xi^k\nh=[\hh\phi_{\nh i}^{1-m}\phi]^{1/2}dx^k$ for $\,k\in\{2i-1,2i\}$. 
Now (\ref{mds}) implies (b), and hence (a).

Next, for 
$\,\theta\hn_k\w=
dx^1\nnh\wedge\ldots\wedge\hs dx^{k-1}\nnh\wedge\hs dx^{k+1}\nnh\wedge
\ldots\wedge\hs dx^{2m}\nh$, the obvious relation 
$\,\zeta_i\w
=-\phi_{\nh i}^{m\hs-1}\chi\hn_1\w\wedge\ldots\wedge\widehat{\chi}_i\w
\wedge\ldots\wedge\chi\hn_m\w$ implies that 
$\,d\hh\zeta_i\w$ equals the combination of $\,\theta\hn_{2i}\w$ and
$\,\theta\hn_{2i-1}\w$ with the coefficients 
$\,-\hs\partial\hn_{2i-1}\w[\hh\phi_{\nh i}^{m\hs-1}]\,$ and 
$\,-\hs\partial\hn_{2i}\w[\hh\phi_{\nh i}^{m\hs-1}]$. As 
$\,\zeta=\zeta_1\w+\ldots+\zeta_m\w$, linear independence of 
the $\,(2m-1)$-forms $\,\theta\hn_1\w\nnh,\dots,\theta\hn_{2m}\w$ yields (c). 
Similarly, $\,d\sy\nnh_i\w$ is the combination of
$\,dx^k\nnh\wedge dx^{2i-1}\nnh\wedge dx^{2i}\nh$, over 
$\,k\notin\{2i-1,2i\}$, with the coefficients 
$\,\partial\nh_k\w[\hh\phi_{\nh i}^{1-m}\phi]$. Linear independence of all
such $\,dx^k\nnh\wedge dx^{2i-1}\nnh\wedge dx^{2i}$ now proves the first
claim in (d). Since, for purely algebraic reasons, 
$\,\phi\hn_i\w=\hs\rho_1\w\nnh\ldots\hs\widehat\rho_i\w\nh\ldots\hs\rho_m\w$ 
for all $\,i\,$ if and only if $\,\phi_{\nh i}^{1-m}\phi=\rho_i^{m-1}\nh$,
the second part of (d) follows.
\end{proof}
\begin{remark}\label{nocnv}Starting from $\,m=3$, Theorem~\ref{duncl} yields
examples in which $\,\zeta\hs$ is closed, but its dual $\,\sy\hs$ is not: 
we may clearly choose $\,\phi\hn_i\w$ as in (c) that do not have separated
variables in the sense of (d).
\end{remark}
\begin{remark}\label{redun}For a nonzero differential $\,(n-2)$-form
on an $\,n$-di\-men\-sion\-al manifold $\,M\nnh$, 
satisfying (\ref{cst}-b) and having the divisibility 
distribution $\,\dz\,$ of fibre dimension\/ $\,q$, so that $\,q\,$ is even 
and $\,2\le q\le n\,$ due to (\ref{qis}), closedness of $\,\my\,$
and in\-te\-gra\-bi\-li\-ty of $\,\dz\,$ do not imply closedness, along
the leaves of $\,\dz$, of the $\,2$-form $\,\sy$ dual to the
in\-di\-vis\-i\-ble factor $\,\zeta\,$ of $\,\my$. Examples realizing all
$\,n,q\,$ with even $\,q$ and $\,n\ge q\ge6\,$ arise from 
Remark~\ref{prdct} applied to a manifold $\,\varSigma\,$ of any dimension
$\,s\ge0\,$ and an open sub\-man\-i\-fold $\,\varPi\hs$ of $\,\bbR\nh^{2m}\nh$,
$\,m\ge3$, with the $\,(2m-2)$-form $\,\zeta$ of Theorem~\ref{duncl}
chosen so as to be closed without closedness of its dual $\,2$-form $\,\sy$
(see Remark~\ref{nocnv}). Here $\,n=2m+s\,$ and $\,q=2m$.
\end{remark}


\end{document}